\documentclass[10pt]{iopart}
\usepackage{iopams}
\usepackage{amsfonts}
\usepackage{epsfig}
\usepackage{subfigure}
\usepackage{mathrsfs}
\usepackage{amsthm}
\usepackage{amssymb}
\usepackage{enumitem}
\usepackage{url}
\usepackage{amstext}
\usepackage{color,soul}
\usepackage{multirow}
\newtheorem{theorem}{Theorem}[section]

\newtheorem{lemma}[theorem]{Lemma}
\newtheorem{proposition}[theorem]{Proposition}
\newtheorem{remark}[theorem]{Remark}
\newtheorem{algorithm}{Algorithm}
\usepackage{epstopdf}
\graphicspath{{fig/}}
\usepackage{lineno,hyperref}
\usepackage{cleveref}
\newtheorem{assumption}{Assumption}[section]

\begin{document}

\title[]{On a posteriori stopping rules of adaptive stochastic heavy ball method for ill-posed problems}

\author{ Ruixue Gu$^1$\footnote{Author to whom any correspondence should be addressed} and Qinian Jin$^2$}
\address{ $^1$ School of Science, Dalian Maritime University, Dalian 116026, PR China\\[1mm]
$^2$ Mathematical Sciences Institute, Australian National University, Canberra ACT
2601, Australia
}
\ead{ruixue\_gu@dlmu.edu.cn, qinian.jin@anu.edu.au}


\begin{abstract}
In this paper we develop a stochastic heavy ball method for solving ill-posed inverse problems. The method updates the iterate using only a randomly selected equation at each iteration step while incorporating a momentum term into the process. To facilitate fast convergence, we propose an adaptive strategy for selecting the step size and the momentum coefficient. Inspired by the spirit of the discrepancy principle, we introduce an {\it a posteriori} stopping rule for our adaptive stochastic heavy ball method. This rule avoids the need to compute residuals of all equations in the system at every iteration or at fixed frequency intervals, thereby enhancing computational efficiency and practicality. Additionally, convex penalty functions are employed to capture the specific features of the desired solutions. Under suitable conditions, we establish almost sure convergence as well as convergence in expectation. Extensive numerical experiments are conducted to evaluate the performance of the proposed method, demonstrating its efficiency and promising potential for solving large-scale ill-posed problems.
\end{abstract}

%

\begin{indented}
\item \noindent{\it Keywords}: {\it a posteriori} stopping rules, stochastic heavy ball method, system of ill-posed equations, convergence
\end{indented}

%
%
%
%
\def\d{\delta}
\def\P{\mathbb{P}}
\def\l{\langle}
\def\r{\rangle}
\def\la{\lambda}
\def\EE{{\mathbb E}}
\def\R{{\mathcal R}}
\def\a{\alpha}
\def\p{\partial}
\def\ep{\varepsilon}
\def\PP{{\mathbb P}}
\def\EE{\mathbb E}
\def\F{{\mathcal F}}
\def\mr{\mathfrak{r}}
\def\m{\mathfrak{m}}
\def\g{\mathfrak{g}}

\def\theequation{\thesection.\arabic{equation}}
\catcode`@=12

\section{\bf Introduction}
\setcounter{equation}{0}

In this paper, we consider ill-posed inverse problems governed by the system
\begin{equation}\label{nonlinear equation}
F_i(x)=y_i, \quad i=1, \cdots, N
\end{equation}
consisting of $N$ equations, where $F_i:\mbox{dom}(F_i)\subset \mathcal{X} \rightarrow {\mathcal{Y}_i}$ are linear or nonlinear operators from Hilbert space $\mathcal{X}$ to Banach spaces ${\mathcal{Y}_i}$ with domain $\mbox{dom}(F_i)$. Such systems (\ref{nonlinear equation}) emerge in a broad range of practical applications, such as various tomography techniques and geophysics \cite{Engl1996Regularization,Hanafy1991,Natterer2001}. We assume that (\ref{nonlinear equation}) has a solution, which might not be unique.  To find the solution with desired feature, one may pick a proper, lower semi-continuous, convex function $\Theta :\mathcal{X} \to \left( { - \infty ,\infty } \right]$ that incorporates {\it a priori} available information of the solution, and select a solution $x^\dag$ of (\ref{nonlinear equation}) such that
\begin{equation}\label{minimum solution}
\fl {D_{{\xi _0}}}\Theta \left( {{x^\dag },{x_0}} \right): = \min_{x \in \mbox{dom}(\Theta) \cap \mathcal{D}} \left\{ {{D_{{\xi _0}}}\Theta \left( {x,{x_0}} \right):{F}_i\left( x \right) = {y_i}}, \, i=1,\cdots,N \right\},
\end{equation}
where $x_0 \in \mbox{dom}\left( {\partial \Theta } \right)$ and $\xi _0  \in \partial \Theta\left( {x_0} \right)$ denote the initial guess. Here $\mathcal{D}: = \bigcap_{i = 1}^N {\mbox{dom}\left( {{F_i}} \right)}  \ne \emptyset $ and ${D_{{\xi _0}}}\Theta\left( {x,{x_0}} \right)$ denotes the Bregman distance  induced by $\Theta$ at $x_0$ in the direction $\xi_0$; see (\ref{bregman distance}) for the definition. Due to the errors in the data acquisition process, the exact data $y: = (y_1,\cdots,y_N)$ is generally unavailable; instead we  only have noisy data $y^\delta := (y_1^\delta,\cdots, y_N^\delta)$ satisfying
 \[
 \left\| {y_i^\delta  - {y_i}} \right\| \le {\delta _i}, \quad i=1, \cdots,N
 \]
 with the noise levels $\delta_i>0$. Such inverse problems are inherently ill-posed in the sense that the solution does not depend continuously on the data. Thus, the reconstruction of the solution $x^\dag$ from noisy data $y^\delta$ necessitates the regularization methods, see \cite{Engl1996Regularization,KNS2008} and the references therein.

A variety of regularization methods have been developed for solving (\ref{nonlinear equation}) in the last three decades; see \cite{Engl1996Regularization,Hanke1995A,KNS2008,Schuster2012Regularization} and the references therein. One of the most prominent methods is the Landweber-type iteration \cite{Bo2011Iterative,JinWang2013}, which reads
\begin{equation}\label{land}
\eqalign{
\xi _{n + 1}^\delta  = \xi _n^\delta  - t_n^\delta \sum\limits_{i = 1}^N {{F_i'}{{\left( {x_n^\delta } \right)}^*}J_r^{{{\cal Y}_i}}\left( {{F_i}\left( {x_n^\delta } \right) - {y_i^\delta} } \right)},\\
x_{n + 1}^\delta  = \arg \mathop {\min }\limits_{x \in \mathcal{X}} \left\{ {\Theta \left( x \right) - \left\langle {\xi _{n + 1}^\delta ,x} \right\rangle } \right\},
}
\end{equation}
where $t_n^\delta$ is the step size and $J_r^{\mathcal{Y}_i}$ denotes the duality mapping of $\mathcal{Y}_i$ corresponding to the gauge function $t \to {t^{r - 1}}$ with  $1<r<\infty$. The convergence of the method (\ref{land}) has been established in \cite{JinWang2013} when it is terminated by {\it a priori} or {\it a posteriori} stopping rules. However, each iteration of (\ref{land}) can be computationally expensive
when $N$ is huge, as it requires solving forward and adjoint problems for all $N$ equations in (\ref{nonlinear equation}), rendering the method impractical in applications.

To tackle the problem, the stochastic gradient descent (SGD) methods, motivated by its success in large-scale optimization \cite{Bottou_2010,Bottou_2018,Jin2025_book,Robbins_1951}, has attracted increasing attention in the inverse problem community. Several studies have been carried out to analyze its convergence properties in the context of ill-posed inverse problems \cite{Gu-Fu-2025,Huang-Jin-Lu-Zhang-2025,JinB_2019,JinB_2020,JinQ_2023,Lu_2022}. In particular the SGD method of the form
\begin{equation}\label{sgdintro}
\eqalign{
\xi _{n + 1}^\delta  = \xi _n^\delta  - t_n^\delta {F_{{i_n}}'}\left( {x_n^\delta } \right)^*J_r^{\mathcal{Y}_{i_n}}\left( {{F_{{i_n}}}\left( {x_n^\delta } \right) - y_{{i_n}}^\delta } \right),\\
x_{n + 1}^\delta  = \arg \mathop {\min }\limits_{x \in \mathcal{X}} \left\{ {\Theta \left( x \right) - \left\langle {\xi _{n + 1}^\delta ,x} \right\rangle } \right\}
}
\end{equation}
has been developed in \cite{Gu-Fu-2025,Huang-Jin-Lu-Zhang-2025,JinQ_2023} for solving (\ref{nonlinear equation}), where the random index $i_n \in \left\{ {1,2, \cdots ,N} \right\}$ is selected uniformly with replacement.
The convergence results of the method have been established in \cite{Gu-Fu-2025, Huang-Jin-Lu-Zhang-2025,JinQ_2023} when the step size is given by
\begin{equation}\label{stepsgd}
\fl t_n^\delta  = \left\{\begin{array}{lll}
\min \left\{\frac{\mu_0 \left\|r_n^\delta\right\|^{p(r - 1)}}{\left\|F_{i_n}'\left(x_n^\delta\right)^* J_r^{{\cal Y}_{i_n}} \left(r_n^\delta\right)\right\|^p}, \mu_1 \right\} \left\|r_n^\delta\right\|^{p - r}, & \mbox{if } \left\| r_n^\delta\right\| > \tau \delta_{i_n}, \\[1mm]
0,  & \mbox{otherwise}
\end{array} \right.
\end{equation}
with $r_n^\delta:=F_{i_n}\left(x_n^\delta\right) - y_{i_n}^\delta$, where $\mu_0>0$, $\mu_1>0$, $p \ge 2$ and $\tau>1$ are suitably chosen constants. At each iteration, the method (\ref{sgdintro}) requires only computing an unbiased estimator $F_{i_n}'\left(x_n^\delta\right)^*J_r^{\mathcal{Y}_{i_n}} \left(F_{i_n}\left(x_n^\delta\right) - y_{i_n}^\delta\right)$, of the full gradient $\frac{1}{N}\sum\nolimits_{i = 1}^N F_i'\left(x_n^\delta\right)^* J_r^{{\cal Y}_i} \left(F_i\left(x_n^\delta\right) - y_i^\delta\right)$, from a randomly selected equation, thereby significantly decreasing the computational cost. However, the presence of stochastic gradient noise causes the SGD method to exhibit pronounced oscillations in the iterates and consequently slows convergence. Although such oscillations could be suppressed by selecting the step size properly as in (\ref{stepsgd}) (see \cite{Gu-Fu-2025,JinQ_2023}), how to incorporate variance reduction techniques to accelerate the method (\ref{sgdintro}) remains an active research topic.

In recent years, various variants of the SGD method have been investigated for solving ill-posed inverse problems, including minibatch \cite{JinQ_2023}, stochastic variance reduced gradient methods \cite{Jin-Zhou-zou-2022,Jin-chen-2025}. More recently, by integrating a momentum term $\beta_n^\delta \left(\xi_n^\delta  - \xi_{n - 1}^\delta\right)$ into the iteration process, a stochastic heavy ball method was proposed in  \cite{Jin-Liu-2025} for solving inverse problems (\ref{nonlinear equation}) with $F_i$ being linear operators. The stochastic heavy ball method in \cite{Jin-Liu-2025} has the form of
\begin{equation}\label{shbintroli}
\eqalign{
\xi_{n + 1}^\delta  = \xi_n^\delta  - t_n^\delta F_{i_n}^*\left(F_{i_n} x_n^\delta - y_{i_n}^\delta\right) + \beta_n^\delta \left(\xi_n^\delta  - \xi_{n - 1}^\delta\right),\\
x_{n + 1}^\delta  = \arg \min_{x \in {\cal X}} \left\{\Theta(x) - \left\langle \xi_{n + 1}^\delta, x\right\rangle \right\},
}
\end{equation}
where $i_n\in \left\{ {1,2, \cdots ,N} \right\}$ is drawn randomly via the uniform distribution, $t_n^\delta$  is the step size and $\beta_n^\delta$ is the momentum coefficient. When $t_n^\delta$  and  $\beta_n^\delta$ are chosen by
\begin{equation}\label{tnbetan}
t_n^\delta  = \frac{\eta_{i_n}}{n + 2} \quad \mbox{and} \quad \beta_n^\delta  = \frac{n}{n + 2}
\end{equation}
with preassigned $0<\eta_{i_n}<1/\|F_{i_n}\|^2$, the convergence of the method (\ref{shbintroli}) has been analyzed in \cite{Jin-Liu-2025} under {\it a priori} stopping rule.
The numerical simulations reported in \cite{Jin-Liu-2025} demonstrate that the inclusion of the momentum term enables the method (\ref{shbintroli}) to outperform the SGD method. Nevertheless, the choice of the parameters $t_n^\delta$ and  $\beta_n^\delta$ is crucial for the performance of the method (\ref{shbintroli}).
It is therefore natural to anticipate that the method (\ref{shbintroli}) can achieve faster convergence through adaptive selection of $t_n^\delta$ and  $\beta_n^\delta$. More importantly, we notice that the most existing work on stochastic iterative methods, including the SGD method mentioned above, is conducted under the {\it a priori} stopping rules. The {\it a priori} stopping rules, however, usually requires extra knowledge of the sought solution, which is typically unavailable in practice. Therefore, it is of practical and theoretical interest to propose an {\it a posteriori} stopping rule and develop corresponding convergence results for stochastic iterative methods.

In this work we will consider solving (\ref{nonlinear equation}) by the following stochastic heavy ball method
\begin{equation}\label{shbintro}
\eqalign{
\xi_{n + 1}^\delta  = \xi_n^\delta  - t_n^\delta F_{i_n}'\left(x_n^\delta\right)^* J_r^{{\cal Y}_{i_n}} \left( F_{i_n} \left(x_n^\delta\right) - y_{i_n}^\delta\right) + \beta_n^\delta\left(\xi_n^\delta - \xi_{n - 1}^\delta \right),\\
x_{n + 1}^\delta  = \arg \min_{x \in {\cal X}} \left\{\Theta(x) - \langle \xi_{n + 1}^\delta, x\rangle\right\},
}
\end{equation}
where the index $i_n$ is drawn randomly from the set $ \left\{1, \cdots, N\right\}$ via a uniform distribution. To facilitate fast convergence, inspired by the work \cite{Jin-2025,Jin-Huang-2024}, the step size $t_n^\delta$ and the momentum coefficient $\beta_n^\delta$ are chosen adaptively to ensure a sufficient decay of the Bregman distance between $x_n^\delta$ and a solution $\hat x$ of (\ref{nonlinear equation}); see section \ref{3themethod} for elaboration. To terminate the iteration (\ref{shbintro}) in an  {\it a posteriori} manner, one may consider the widely used discrepancy principle
\begin{equation*}
\left\|F_i\left(x_{n_\delta}^\delta \right)-y_i^\delta \right\| \le \tau \delta_i,  \quad \forall i=1,\cdots,N.
\end{equation*}
However, it is worth noting that directly computing  the residuals $\left\|F_i\left(x_n^\delta\right)-y_i^\delta \right\|$ for all $i$ at every iteration, or with some frequency, can result in prohibitive overhead \cite{Huang-Jin-Lu-Zhang-2025}. To address this issue, we adopt the following implementable and effective strategy: initialize a set $I_0(y^\delta):=\{1, \cdots, N\}$, which records the equations that might not satisfy (\ref{approximatesolution}). For the randomly selected index $i_n$, if $\left\| {{F_{{i_n}}}\left( {x_n^\delta } \right) - y_{{i_n}}^\delta } \right\| \le \tau {\delta _{{i_n}}}$ and $\beta_n^\delta$=0---indicating that the discrepancy principle has been satisfied for the $i_n$-th equation and the iteration makes no further contribution---then we update $I_{n+1}(y^\delta):=I_{n}(y^\delta)\backslash \left\{i_n \right\}$ and continue the iteration using the remaining indices in $I_{n+1}(y^\delta)$. Otherwise, the set $I_{n+1}(y^\delta)$ is reset to $\{1, \cdots, N\}$. The algorithm terminates when $I_{n_\delta}(y^\delta)$  becomes empty. A detailed description of the method (\ref{shbintro}) incorporating this stopping rule is provided in Algorithm \ref{shbnoisy}. 

The primary contributions of this work are the following:
\begin{itemize}
\item  We propose a stochastic heavy ball method (\ref{shbintro}), and, by incorporating the spirit of the discrepancy principle, develop an {\it a posteriori} stopping rule. This stopping rule eliminates the need to compute the residuals of all equations at every iteration or at fixed frequency intervals, making the method particularly efficient for large-scale ill-posed problems.

\item  We establish the almost sure convergence and the convergence in expectation of the method (\ref{shbintro}) under the proposed {\it a posteriori} stopping rule.
Note that the analysis in \cite{Jin-Liu-2025} relies on the {\it a priori} stopping rule and the specific structure of the method when the operators $F_i$ are linear and the parameters $t_n^\delta$ and $\beta_n^\delta$ are chosen according to (\ref{tnbetan}), and thus is no longer applicable in our setting.
Due to the randomness of the stopping index $n_\delta$ produced by Algorithm \ref{shbnoisy} and the adaptive choices of $t_n^\delta$ and $\beta_n^\delta$, our convergence analysis is highly challenging. The analysis is carried out by first showing that $n_\delta$ is finite {\it almost surely}, and then constructing an event of probability one, on which the method with exact data converges to the solution along any sample path. This, combined with the stability property, paves the way for establishing the convergence result in the noisy data case.

\item  Additionally, in section \ref{sect4}, we investigate a mini-batch variant of the stochastic heavy ball method. As a byproduct, we also derive the convergence results of the SGD method when it is terminated by our proposed stopping criterion. The numerical experiments in section \ref{3numbericalex} validate the theoretical findings and the superiority of the method (\ref{shbintro}) over the SGD method.
\end{itemize}
This paper is organized as follows. Section \ref{preliminaries} gives a brief exposition of convex analysis that are crucial in the theoretical analysis of the method. In section \ref{3themethod}, we discuss the choice of the step size and the momentum coefficient, and design an {\it a posteriori} stopping rule for the stochastic heavy ball method. Then, under suitable conditions, we establish both almost sure convergence and convergence in expectation of the method. In section \ref{sect4} we present a minibatch variant of the stochastic heavy ball method discussed in section \ref{3themethod}. Finally, in section \ref{3numbericalex}, we report numerical simulations to demonstrate the effectiveness of the method.

\section{Preliminaries}\label{preliminaries}

In this section, we collect some notions and basic facts related to convex analysis and Banach spaces; for more information we refer the reader to \cite{Cioranescu1990Geometry,Jin2025_book,Zalinescu2002Convex}.

Let $\mathcal{X}$ be a Banach space with its dual space denoted by ${\mathcal{X}^*}$; when $\mathcal{X}$ is a Hilbert space, its dual space ${\mathcal{X}^*}$ is identified with $\mathcal{X}$ itself. For $x \in \mathcal{X}$, and $\xi \in \mathcal{X}^*$, we denote by $\langle \xi, x\rangle  = \xi(x)$ the duality pairing. For a bounded linear operator $A:\mathcal{X} \to \mathcal{Y}$ between two Banach spaces $\mathcal{X}$ and $\mathcal{Y}$, let $A^*:\mathcal{Y}^* \to \mathcal{X}$ and ${\rm Ran}(A)$ denote its adjoint and range space respectively.

For a convex function $f :\mathcal{X} \to \left( { - \infty ,\infty } \right]$, its effective domain is defined by $\text{dom}\left( f \right): = \left\{ {x \in \mathcal{X}:f \left( x \right) < \infty } \right\}$. We call $f$ proper if $\text{dom}(f) \ne \emptyset$. For any $x \in \mathcal{X}$, we define
\begin{equation*}
\partial f(x): = \left\{\xi \in \mathcal{X}^*: f(\bar x) - f(x) - \langle \xi, \bar x - x\rangle  \ge 0,~\forall~\bar x \in \mathcal{X} \right\},
\end{equation*}
which is called the subdifferential of $f$ at $x$. Each $\xi \in \partial f(x)$ is called a subgradient of $f $ at $x$. Note that $\p f$ defines a set-valued mapping from ${\mathcal X}$ to ${\mathcal X}^*$. The domain and graph of $\p f$ are
defined respectively as
\begin{equation*}
\mbox{dom}(\p f) := \{x\in \mbox{dom}(f): \p f(x)\ne \emptyset \}
\end{equation*}
and
\begin{equation*}
\mbox{graph}(\p f)
:= \left\{(x, \xi)\in {\mathcal X}\times {\mathcal X}^*: x\in \mbox{dom}(\p f) \mbox{ and } \xi \in \p f(x)\right\}.
\end{equation*}
Given any $(x, \xi) \in \text{graph}(\partial f)$, the quantity
\begin{equation}\label{bregman distance}
D_{\xi} f(\bar x, x): = f(\bar x) - f(x) - \langle \xi, \bar x - x\rangle, \quad \forall \bar x \in \mathcal{X}
\end{equation}
is called the Bregman distance induced by $f$ at $x$ in the direction $\xi$. Clearly, the Bregman distance is always nonnegative and satisfies the identity
\begin{equation}\label{2.9}
D_{\xi_2} f(x, x_2) - D_{\xi_1} f(x, x_1) = D_{\xi_2} f(x_1, x_2) + \l \xi_2 - \xi_1, x_1 - x\r
\end{equation}
for all $x\in \mathcal{X}$ and $(x_1, \xi_1), (x_2, \xi_2) \in \mbox{graph}(\p f)$.

A proper function $f :\mathcal{X} \to \left( { - \infty ,\infty } \right]$ is said to be strongly convex if there exists a constant ${\sigma} > 0$ such that
\begin{equation}\label{pconvex}
f \left( {tx + \left( {1 - t} \right)\bar{x}} \right) + {\sigma}t\left( {1 - t} \right){\left\| {x - \bar{x}} \right\|^2} \le tf \left( x \right) + \left( {1 - t} \right)f \left( \bar{x} \right)
\end{equation}
for all $x,\bar{x} \in \mathcal{X}$ and $0 \le t \le 1$. It is easy to see that if $f : \mathcal{X} \to (-\infty, \infty]$
is strongly convex in the sense of (\ref{pconvex}), then
\begin{equation}\label{theta p-convex}
D_{\xi} f(\bar x, x) \ge \sigma \|\bar{x} - x\|^2
\end{equation}
for all $\bar x\in {\mathcal X}$ and $(x, \xi) \in \mbox{graph}(\p f)$.

\begin{lemma}\label{lem1}
Let $f: \mathcal{X} \to (-\infty, \infty]$ be a strongly convex function in the sense of (\ref{pconvex}). Then for any $(x_1, \xi_1), (x_2, \xi_2) \in {\rm graph}(\p f)$ there holds
\begin{equation*}
D_{\xi_1} f(x_2, x_1) \le \frac{1}{4\sigma} \|\xi_2 -\xi_1\|^2.
\end{equation*}
\end{lemma}

\begin{proof}
By the definition of Bregman distance, we have
\begin{equation*}
D_{\xi_1} f(x_2, x_1) + D_{\xi_2} f(x_1, x_2) = \l \xi_1 - \xi_2, x_1 - x_2\r.
\end{equation*}
By the strong convexity of $f$ we then have $D_{\xi_2} f(x_1, x_2) \ge \sigma \|x_1 - x_2\|^2$. Therefore
\begin{eqnarray*}
\fl D_{\xi_1} f(x_2, x_1) + \sigma \|x_1 -x_2\|^2
& \le \l \xi_1 - \xi_2, x_1 -x_2\r \le \|\xi_1 - \xi_2\| \|x_1 -x_2\| \\
& \le \frac{1}{4\sigma} \|\xi_1 - \xi_2\|^2 + \sigma \|x_1 - x_2\|^2,
\end{eqnarray*}
which implies the desired estimate.
\end{proof}

\begin{lemma}\label{minixi}
If $f :\mathcal{X} \to (-\infty, \infty]$ is a proper, lower semi-continuous and strongly convex function in the sense of (\ref{pconvex}), then for any $\xi \in \mathcal{X}^*$, the minimization problem
\[
\min_{z \in \mathcal{X}} \left\{f(z) - \langle \xi, z\rangle \right\}
\]
has a unique minimizer which is denoted by $x_\xi$. Moreover $\xi \in \partial f(x_\xi)$ and
\begin{equation}\label{2.11}
\|x_\xi - x_{\tilde \xi}\| \le \frac{1}{2\sigma} \|\xi - \tilde \xi\|, \quad \forall \xi, \tilde \xi \in {\mathcal X}^*.
\end{equation}
\end{lemma}

\begin{proof}
Since $f$ is proper, lower semi-continuous and strongly convex, the existence and uniqueness of $x_\xi$ follows readily. By the first order optimality there holds $\xi \in \p f(x_\xi)$. Moreover, it follows from (\ref{theta p-convex}) that
\begin{eqnarray*}
2\sigma \|x_\xi - x_{\tilde \xi}\|^2
& \le D_\xi f(x_{\tilde \xi}, x_\xi) + D_{\tilde \xi} f(\xi, \tilde \xi)
= \l \xi - \tilde \xi, x_\xi - x_{\tilde \xi}\r \\
& \le \|\xi - \tilde \xi\| \|x_\xi - x_{\tilde \xi}\|,
\end{eqnarray*}
which implies (\ref{2.11}).
\end{proof}

On a Banach space $\mathcal{Y}$, we consider for $1 < r < \infty$ the convex function $y \to \|y\|^r/r$ whose subdifferential at $y$ given by
\[
J_r^\mathcal{Y}\left( y \right): = \left\{ {\xi  \in {\mathcal{Y}^*}:\left\| \xi  \right\| = {{\left\| y \right\|}^{r - 1}}~\text{and}~\left\langle {\xi ,y} \right\rangle  = {{\left\| y \right\|}^r}} \right\}
\]
which gives the duality mapping $J_r^\mathcal{Y}:\mathcal{Y} \to {2^{{\mathcal{Y}^*}}}$  with gauge function $t \to t^{r - 1}$. When $\mathcal{Y}$ is uniformly smooth in the sense that the modulus of smoothness
\[
{\rho _\mathcal{Y}}\left( \varepsilon  \right): = \frac{1}{2}\sup \left\{ {\left\| {\bar y + y} \right\| + \left\| {\bar y - y} \right\| - 2:\left\| \bar y \right\| = 1,\left\| y \right\| \le \varepsilon } \right\}
\]
satisfies $\lim_{\varepsilon \to 0} \rho_\mathcal{Y}(\varepsilon)/ \varepsilon = 0$, the duality mapping $J_r^\mathcal{Y}$ with $1<r<\infty$ is single valued and uniformly continuous on bounded sets.

\section{The adaptive stochastic heavy ball method}\label{3themethod}

In this section, we detail the selection of  the step size and momentum coefficient for the stochastic heavy ball method (\ref{shbintro}), and design an {\it a posteriori} stopping rule. We then establish the almost sure convergence and convergence in expectation of the method under this stopping criterion.
Throughout, we make the following assumptions on $\Theta$ and the operators $F_i$, where $B_\rho(x_0): = \left\{x \in \mathcal{X}: \|x - x_0\| \le \rho \right\}$ and $\mathcal{D}: = \bigcap_{i = 1}^N \text{dom}(F_i) \ne \emptyset$.

\begin{assumption}\label{5asumption theta}
$\Theta :\mathcal{X} \to \left( { - \infty ,\infty } \right]$ is a  proper, lower semi-continuous, strongly convex function satisfying (\ref{theta p-convex}) for some constant $\sigma>0$.
\end{assumption}

\begin{assumption}\label{5asumption operator}
\begin{itemize}
\item [{\rm (a)}] $\mathcal{X}$ is a Hilbert space and each $\mathcal{Y}_i$, $i=1,2,\ldots,N$, is a uniformly smooth Banach space.

\item [{\rm (b)}] There exists $\rho>0$ such that ${B_{2\rho }}\left( {{x_0}} \right) \subset \mathcal{D}$ and (\ref{nonlinear equation}) has a solution $x_*$ in $ {\rm dom}\left( \Theta  \right)$ satisfying
   ${D_{{\xi _0}}}\Theta \left( { x_*,{x_0}} \right) \le {\sigma}{\rho ^2}$.

\item [{\rm (c)}] For each $i = 1, \cdots, N$, there exists a family of bounded linear operators $\{L_i(x): \mathcal{X} \to \mathcal{Y}\}_{x\in B_{2\rho}(x_0)}$ such that $x \to L_i(x)$ is continuous on $B_{2\rho}(x_0)$; moreover
\[
\left\| L_i(x)\right\| \le B_0, \quad \forall x\in B_{2\rho}(x_0) \mbox{ and } i = 1, \cdots, N
\]
for some constant $B_0>0$ and there is $0 \le \eta  < 1$ such that
\begin{equation}\label{TCC}
\left\|F_i(x) - F_i(\bar x) - L_i(x) (x - \bar x)\right\| \le \eta \left\|F_i(x) - F_i(\bar x)\right\|
\end{equation}
for all $x,\bar x  \in B_{2\rho}(x_0)$ and $i=1, \cdots, N$.
\end{itemize}
\end{assumption}

The conditions in Assumption \ref{5asumption theta} and Assumption \ref{5asumption operator} are standard in the convergence analysis of regularization methods for ill-posed inverse problems \cite{Hanke1995A,JinWang2013,KNS2008}. They in particular guarantee the existence and uniqueness of a solution $x^\dag$ of (\ref{nonlinear equation}) satisfying (\ref{minimum solution}), see \cite[Lemma 3.2]{JinWang2013}. According to Assumption \ref{5asumption operator} (b), there holds
\begin{equation}\label{6xdag}
D_{\xi _0}\Theta(x^\dag, x_0) \le D_{\xi _0}\Theta(x_*, x_0) \le \sigma \rho ^2
\end{equation}
which together with Assumption \ref{5asumption theta} yields $\left\|x^\dag - x_0 \right\| \le \rho$, i.e., $x^\dag \in B_\rho(x_0)\cap \text{dom}(\Theta)$.

Next we will consider the  stochastic heavy ball method  (\ref{shbintro})  and discuss how to choose the step-size $t_n^\delta$ and the momentum coefficient $\beta_n^\delta$. We take $\xi_{-1}^\delta=\xi_0^\delta=\xi_0$ and $x_0^\delta=x_0$. For $n \ge 0$, the index $i_n$ is picked randomly from $\left\{ {1, \cdots, N} \right\}$ via a uniform distribution. For ease of notation, we denote
\begin{equation*}
\mr_n^{\delta}:= F_{i_n}\left(x_n^\delta\right) - y_{i_n}^\delta, \quad
\g_n^\delta:=L_{i_n}\left(x_n^\delta\right)^* J_r^{{\cal Y}_{i_n}} \left(\mr_n^\delta\right), \quad
\m_n^\d := \xi_n^\d - \xi_{n-1}^\d,
\end{equation*}
then the method (\ref{shbintro}) can be formulated as
\begin{equation}\label{shb}
\eqalign{
\xi_{n + 1}^\delta  = \xi_n^\delta - t_n^\delta \g_n^\delta + \beta_n^\delta \m_n^\delta, \\
x_{n + 1}^\delta = \arg \min_{x \in {\cal X}} \left\{\Theta(x) - \langle \xi_{n + 1}^\delta, x \rangle \right\}.
}
\end{equation}
In order for the method (\ref{shb}) to have fast convergence, the parameters $t_n^\delta$ and $\beta_n^\delta$ should be selected carefully. Let $\hat x$ be any solution of (\ref{nonlinear equation}) in ${B_{2\rho} }\left( {{x_0}} \right)\cap {\rm dom}\left( \Theta  \right)$, we consider the decay of the Bregman distance
${D_{\xi _n^\delta }}\Theta \left( {\hat x,x_n^\delta } \right)$ with respect to $n$, where ${D_{\xi _n^\delta }}\Theta \left( {\hat x,x_n^\delta } \right)$ measures the closeness of $x_n^\delta$ to $\hat x$. From (\ref{2.9}) and Lemma \ref{lem1} it follows that
\begin{eqnarray*}
D_{\xi_{n + 1}^\delta}\Theta \left(\hat x,x_{n + 1}^\delta\right) - D_{\xi_n^\delta }\Theta \left(\hat x, x_n^\delta \right) \\
= D_{\xi_{n+1}^\d} \Theta(x_n^\d, x_{n+1}^\d) + \l \xi_{n+1}^\d - \xi_n^\d, x_{n}^\d - \hat x\r \\
\le \frac{1}{{4\sigma }}{\left\|\m_{n + 1}^\delta\right\|^2} + \left\langle \m_{n + 1}^\delta, x_n^\delta - \hat x \right\rangle.
\end{eqnarray*}
By the definition of $\xi_{n+1}^\delta$ and the polarization identity, we further have
\begin{eqnarray*}
& D_{\xi_{n + 1}^\delta}\Theta \left(\hat x, x_{n + 1}^\delta\right) - D_{\xi_n^\delta}\Theta\left(\hat x, x_n^\delta\right)\\
& \le \frac{1}{4\sigma} \left\|- t_n^\delta \g_n^\delta + \beta_n^\delta \m_n^\delta\right\|^2 + \left\langle - t_n^\delta \g_n^\delta  + \beta _n^\delta \m_n^\delta, x_n^\delta - \hat x \right\rangle \\
& = \frac{1}{4\sigma} \left(t_n^\delta\right)^2 \left\|\g_n^\delta\right\|^2
+ \frac{1}{4\sigma } \left(\beta_n^\delta\right)^2 \left\|\m_n^\delta\right\|^2 - \frac{1}{2\sigma} t_n^\d \beta_n^\d \l \g_n^\d, \m_n^\d\r \\
& \quad \, + t_n^\delta \left\langle \g_n^\delta, \hat x-x_n^\delta \right\rangle + \beta_n^\delta \left\langle \m_n^\delta, x_n^\delta  - \hat x \right\rangle.
\end{eqnarray*}
Assume that $x_n^\delta \in B_{2\rho}(x_0)$, utilizing (\ref{TCC}) in Assumption \ref{5asumption operator}(c), $\left\|y_{i_n}^\delta  - y_{i_n}\right\| \le \delta_{i_n}$ and the property of the duality mapping $J_r^{\mathcal{Y}_{i_n}}$, we can obtain
\begin{eqnarray}\label{undelta}
\left\langle \g_n^\delta, \hat x - x_n^\delta \right\rangle
& = \left\langle J_r^{{\cal Y}_{i_n}} \left(\mr_n^\delta\right), y_{i_n}^\delta - F_{i_n} \left(x_n^\delta\right) \right\rangle + \left\langle J_r^{{\cal Y}_{i_n}} \left(\mr_n^\delta\right), y_{i_n} - y_{i_n}^\d \right\rangle \nonumber\\
& \quad \, - \left\langle J_r^{{\cal Y}_{i_n}} \left(\mr_n^\delta\right), y_{i_n} - F_{i_n}\left(x_n^\delta\right) - L_{i_n}\left(x_n^\delta\right) \left(\hat x - x_n^\delta\right) \right\rangle \nonumber\\
& \le - \left\|\mr_n^\delta\right\|^r + \left\|\mr_n^\delta\right\|^{r - 1} \left(\delta _{i_n} + \eta \left\| y_{i_n} - F_{i_n}\left(x_n^\delta\right)\right\| \right) \nonumber\\
& \le - \left\|\mr_n^\delta\right\|^r + \left\|\mr_n^\delta\right\|^{r - 1} \left(\left(1 + \eta\right)\delta_{i_n} + \eta \left\|\mr_n^\delta\right\|\right) \nonumber \\
& = - (1 - \eta) \left\|\mr_n^\delta\right\|^r + (1 + \eta) \delta_{i_n} \left\|\mr_n^\delta\right\|^{r - 1}.
\end{eqnarray}
Therefore, we derive the estimate
\begin{eqnarray}\label{decay}
& D_{\xi_{n + 1}^\delta}\Theta \left(\hat x, x_{n + 1}^\delta\right) - D_{\xi_n^\delta}\Theta \left(\hat x, x_n^\delta\right) \nonumber \\
& \le \frac{1}{4\sigma} \left(t_n^\delta\right)^2 \left\|\g_n^\delta\right\|^2 - (1 - \eta) t_n^\delta \left\| \mr_n^\delta \right\|^r + (1 + \eta) \delta_{i_n} t_n^\delta \left\|\mr_n^\delta\right\|^{r - 1} \nonumber\\
& \quad \, + \frac{1}{4\sigma} \left(\beta_n^\delta\right)^2 \left\|\m_n^\delta\right\|^2 - \frac{1}{2\sigma} t_n^\d \beta_n^\d \l \g_n^\d, \m_n^\d\r + \beta_n^\delta \left\langle \m_n^\delta, x_n^\delta - \hat x \right\rangle.
\end{eqnarray}
To promote rapid convergence of the method  (\ref{shb}), inspired by the work \cite{Jin-2025,Jin-Huang-2024}, one could choose $t_n^\delta$ and $\beta_n^\delta$ such that the right hand side of (\ref{decay}) is as small as possible. We first determine $t_n^\delta$ by assuming $\beta_n^\delta=0$. In this case, only the first three terms on the right hand side of (\ref{decay}) remain. To ensure that the sum of these three terms is negative,  $t_n^\delta$ should be chosen to satisfy
\[
0 \le t_n^\delta < \frac{4\sigma \left((1 - \eta)\left\|\mr_n^\delta\right\| - (1 + \eta)\delta_{i_n}\right) \left\| \mr_n^\delta\right\|^{r - 1}}{\left\|\g_n^\delta\right\|^2}
\]
provided that $\left\|\mr_n^\delta\right\| > \tau \delta_{i_n}$ with given $\tau>(1 + \eta)/(1-\eta)$. This leads us to choose
\begin{equation}\label{tndelta}
\fl t_n^\delta  = \left\{ \begin{array}{lll}
\min \left\{\frac{\mu_0 \left((1 - \eta)\left\|\mr_n^\delta\right\| - (1 + \eta) \delta_{i_n}\right) \left\| \mr_n^\delta\right\|^{r - 1}}{\left\|\g_n^\delta \right\|^2}, \mu_1 \left\|\mr_n^\delta\right\|^{2 - r} \right\}, & \text{if } \left\|\mr_n^\delta\right\| > \tau \delta_{i_n},\\[1mm]
0, & \text{otherwise}
\end{array} \right.
\end{equation}
with $0<\mu_0< 4 \sigma$  and $\mu_1>0$. Substituting such $t_n^\delta$ into (\ref{decay}), there follows
\begin{eqnarray*}
& D_{\xi_{n + 1}^\delta}\Theta \left(\hat x, x_{n + 1}^\delta\right) - D_{\xi _n^\delta}\Theta \left(\hat x, x_n^\delta\right)\\
& \le - \left(1 - \frac{\mu _0}{4\sigma}\right) t_n^\delta \left((1 - \eta) \left\|\mr_n^\delta\right\| - (1 + \eta) \delta_{i_n}\right) \left\|\mr_n^\delta\right\|^{r - 1}\\
& \quad \, + \frac{1}{4\sigma} \left(\beta _n^\delta\right)^2 \left\|\m_n^\delta\right\|^2 - \frac{1}{2\sigma} t_n^\d \beta_n^\d \l \g_n^\d, \m_n^\d\r + \beta_n^\delta \left\langle \m_n^\delta, x_n^\delta  - \hat x \right\rangle.
\end{eqnarray*}
Next we consider the selection of $\beta_n^\delta$. Before that, we first give an estimate on the term
$$
\gamma _n^\delta : = \left\langle \m_n^\delta, x_n^\delta  - \hat x \right\rangle,
$$
which is uncomputable in practice because it depends on the unknown solution $\hat x$. By virtue of the definition of $\xi_n^\delta$ and (\ref{undelta}),  the term $\gamma_n^\delta$ can be estimated by
\begin{eqnarray*}
\gamma _n^\delta
& = \left\langle \m_n^\delta, x_n^\delta - x_{n - 1}^\delta \right\rangle + \left\langle \m_n^\delta, x_{n - 1}^\delta  - \hat x \right\rangle \\
& =\left\langle \m_n^\delta, x_n^\delta - x_{n - 1}^\delta \right\rangle + \left\langle - t_{n - 1}^\delta \g_{n - 1}^\delta  + \beta_{n - 1}^\delta \m_{n - 1}^\delta, x_{n - 1}^\delta  - \hat x \right\rangle\\
&\le \left\langle \m_n^\delta, x_n^\delta - x_{n - 1}^\delta \right\rangle - (1 - \eta) t_{n - 1}^\delta \left\| \mr_{n - 1}^\delta\right\|^r \\
& \quad \, + (1 + \eta) \delta_{i_{n - 1}} t_{n - 1}^\delta \left\|\mr_{n - 1}^\delta\right\|^{r - 1} + \beta_{n - 1}^\delta \gamma_{n - 1}^\delta.
\end{eqnarray*}
This motivates us to define $\left\{ {\tilde \gamma _n^\delta } \right\}$ by setting $\tilde \gamma _0^\delta=0$ and, for $n\ge 1$,
\begin{eqnarray}\label{tildegamma}
\tilde \gamma_n^\delta  &= \left\langle \m_n^\delta, x_n^\delta - x_{n - 1}^\delta\right\rangle - (1 - \eta) t_{n - 1}^\delta \left\|\mr_{n - 1}^\delta\right\|^r \nonumber \\
& \quad \, + (1 + \eta) \delta_{i_{n - 1}} t_{n - 1}^\delta \left\|\mr_{n - 1}^\delta\right\|^{r - 1} + \beta_{n - 1}^\delta \tilde \gamma_{n - 1}^\delta.
\end{eqnarray}
Suppose that $x_k^\delta \in B_{2\rho}(x_0)$ and $\beta_k^\d \ge 0$ for all $ 0\le k<n$, then one may use an induction argument to show that $\gamma_k^\delta \le \tilde \gamma_k^\delta$ for all $ 0\le k \le n$. Therefore
\begin{eqnarray}\label{betan}
& D_{\xi_{n + 1}^\delta}\Theta \left(\hat x, x_{n + 1}^\delta\right) - D_{\xi _n^\delta}\Theta \left(\hat x, x_n^\delta\right) \nonumber\\
& \le  - \left(1 - \frac{\mu_0}{4\sigma}\right) t_n^\delta \left((1 - \eta) \left\|\mr_n^\delta\right\| - (1 + \eta) \delta_{i_n}\right) \left\|\mr_n^\delta\right\|^{r - 1} \nonumber \\
& \quad \, + \frac{1}{4\sigma} \left(\beta_n^\delta\right)^2 \left\|\m_n^\delta\right\|^2 - \frac{1}{2\sigma} t_n^\d \beta_n^\d \l \g_n^\d, \m_n^\d \r + \beta_n^\delta \gamma_n^\delta \nonumber \\
& \le  - \left(1 - \frac{\mu_0}{4\sigma}\right) t_n^\delta \left((1 - \eta) \left\|\mr_n^\delta\right\| - (1 + \eta) \delta_{i_n} \right) \left\|\mr_n^\delta\right\|^{r - 1} \nonumber \\
& \quad \, + \frac{1}{4\sigma} \left(\beta _n^\delta\right)^2 \left\|\m_n^\delta\right\|^2 - \frac{1}{2\sigma} t_n^\d \beta_n^\d \l\g_n^\d, \m_n^\d\r + \beta _n^\delta \tilde \gamma _n^\delta.
\end{eqnarray}
A natural approach is to choose $\beta_n^\delta$ such that the right hand side of (\ref{betan}) is minimized over the interval $\left[ {0,\beta } \right]$, where $0<\beta\le \infty$ is a user chosen parameter. This gives
\begin{equation}\label{betagamma}
\beta_n^\delta=\min \left\{\max\left\{0, \frac{t_n^\d \l \g_n^\d, \m_n^\d\r - 2\sigma \tilde \gamma_n^\delta}{\left\| \m_n^\delta\right\|^2} \right\}, \beta \right\}
\end{equation}
in case $\m_n^\delta \ne 0$. Note that this $\beta_n^\delta$ is the minimizer of the function $s \to {h_n}\left( s \right)$ over $\left[ {0,\beta } \right]$, where
\[
{h_n}(s) = \frac{1}{4\sigma} s^2 \|\m_n^\delta\|^2
- \frac{1}{2\sigma} s t_n^\d  \l \g_n^\d, \m_n^\d\r + s \tilde \gamma_n^\delta.
\]
When using (\ref{betagamma}) to compute $\beta_n^\d$, stability considerations naturally require that the denominator $\|\m_n^\d\|^2$ not be too small. Moreover, we would like $\beta_n^\d$ to remain suitably large so that the momentum term can make a meaningful contribution to acceleration. To address these issues, we compute $\beta_n^\d$ only when
\begin{equation}\label{beta0}
\|\m_n^\d\| > \upsilon_0 \d_{i_n} \quad \mbox{and} \quad
\tilde \gamma_n^\d - \frac{t_n^\d}{2\sigma} \l \g_n^\d, \m_n^\d \r < - \upsilon_1 \d_{i_n} \|\m_n^\d\|^2,
\end{equation}
where $\upsilon_0$ and $\upsilon_1$ are user-defined small positive constants. In this case, we modify $\beta_n^\d$ as follows:
\[
\beta_n^\delta=\min \left\{ \frac{t_n^\d \l \g_n^\d, \m_n^\d \r - 2\sigma \tilde \gamma_n^\delta}{\|\m_n^\delta\|^2}, \beta \right\}
\]
and set $\beta_n^\d = 0$ whenever (\ref{beta0}) is not satisfied.

With the above adaptive strategy for selecting $t_n^\delta$ and $\beta_n^\delta$ in place, we now consider the stochastic heavy ball method (\ref{shbintro}). To obtain a reasonable approximate solution, the iteration must be terminated appropriately. One of the most widely used {\it a posteriori} stopping rules is the discrepancy principle, which determines a stopping index $n_\d$ such that
\begin{equation}\label{approximatesolution}
\left\| F_i\left(x_{n_\delta}^\delta \right)-y_i^\delta \right\| \le \tau \delta_i,  \quad \forall i=1,\cdots,N
\end{equation}
for the first time. However, it is important to note that directly evaluating all the residuals
$\left\| F_i\left(x_n^\delta \right) - y_i^\delta \right\|$ for all $i$ at each iteration, or even at fixed intervals, can result in prohibitive computational costs. Hence, it is essential to incorporate the discrepancy principle into (\ref{shbintro}) in a way that is both implementable and computationally efficient.
To this end, we adopt the following effective strategy: initialize a set $I_0(y^\delta):=\left\{ {1,\cdots,N} \right\}$. At iteration $n$, if $\left\| \mr_n^\delta \right\| \le \tau \delta _{i_n}$ and $\beta_n^\delta=0$, then the index $i_n$ is removed from $I_n(y^\delta)$, i.e., $I_{n+1}(y^\delta):=I_{n}(y^\delta)\setminus \{i_n\}$; otherwise, the set is reset to $\left\{ {1,\cdots,N} \right\}$. The algorithm terminates once $I_{n_\delta}(y^\delta)$ becomes empty, and $x_{n_\delta}^\delta$ is then taken as the  approximate solution. The details are summarized in Algorithm~\ref{shbnoisy} below.

\begin{algorithm}[Adaptive SHB method with noisy data]\label{shbnoisy}
Let $\mu_0>0$, $\mu_1>0$, $\tau>1$, $\upsilon_0 >0$, $\upsilon_1>0$, $0<\beta\le \infty$. Pick an initial guess $x_0 \in {\rm dom}\left( {\partial  \Theta } \right)$ and $\xi _0  \in \partial \Theta\left( {x_0} \right)$. Let $\xi_{-1}^\delta=\xi_0^\delta:=\xi_0$ and $x_0^\delta:=x_0$. Set $I_0(y^\delta):=\left\{ {1, \cdots, N} \right\}$ and $\tilde \gamma _0^\delta:=0$. For $n\ge 0$, do the following:

\begin{itemize}
\item[\emph{(i)}]  Sample an index  ${i_n} \in \left\{ {1, \cdots, N} \right\}$ at random via a uniform distribution;

\item[\emph{(ii)}]  Compute $\mr_n^{\delta}:=F_{i_n} \left(x_n^\delta\right) - y_{i_n}^\delta$, $\g_n^\delta:= L_{i_n}\left(x_n^\delta\right)^* J_r^{{\cal Y}_{i_n}} \left(\mr_n^\delta\right)$, and determine $t_n^\delta$ by (\ref{tndelta});

\item[\emph{(iii)}] Set $\m_n^\d := \xi_n^\d - \xi_{n-1}^\d$, compute $\tilde \gamma _n^\delta $ by (\ref{tildegamma})
and determine $\beta_n^\d$ by
\[
\beta _n^\delta  = \left\{ \begin{array}{lll}
\min \left\{ \frac{t_n^\d \l \g_n^\d, \m_n^\d\r - 2 \sigma\tilde \gamma_n^\delta}{\|\m_n^\delta\|^2},\beta \right\}, & \emph{if (\ref{beta0}) holds}, \\
0, & \emph{otherwise};
\end{array} \right.
\]

\item[\emph{(iv)}] Update $\xi _{n + 1}^\delta$ and $x_{n + 1}^\delta$ by
\begin{eqnarray*}
\xi_{n + 1}^\delta  = \xi_n^\delta  - t_n^\delta \g_n^\delta + \beta_n^\delta \m_n^\delta, \\[1mm]
x_{n + 1}^\delta  = \arg \min_{x \in {\cal X}} \left\{\Theta(x) - \left\langle \xi_{n + 1}^\delta, x \right\rangle \right\};
\end{eqnarray*}

\item[\emph{(v)}] Set
\[
{I_{n + 1}}({y^\delta }): = \left\{ \begin{array}{lll}
{I_n}({y^\delta })\backslash \left\{ {{i_n}} \right\}, & \emph{if } t_n^\d =0 \emph{ and } \beta_n^\d =0, \\[2mm]
\left\{ {1,\cdots,N} \right\}, & \emph{otherwise};
\end{array} \right.
\]

\item[\emph{(vi)}] Let ${n_\delta}$ be the first integer such that $I_{n_\delta}(y^\delta)=\emptyset$ and use $x_{n_\delta}^\delta$ as  an approximate solution.
\end{itemize}
\end{algorithm}

\begin{remark}\label{remarkal}
{\rm
At each iteration, the primary computational cost of Algorithm \ref{shbnoisy} arises from evaluating $\mr_n^\delta$ and $\g_n^\delta$; the effort required to calculate $t_n^\delta$,  $\tilde \gamma _n^\delta $ and $\beta_n^\delta$ is comparatively  negligible.
The discrepancy principle (\ref{approximatesolution}) is embedded in the set $I_n(y^\delta)$, which avoids the expense of calculating $\|F_i\left(x_n^\delta \right)-y_i^\delta\|$ for all $i\in \left\{ {1,\cdots,N} \right\}$ and therefore is quite efficient.
}
\end{remark}

\begin{remark}
{\rm In step (i) of Algorithm \ref{shbnoisy}, if the selected index $i_n$, $n\ge 1$, lies outside $I_n(y^\d)$, then $\xi_{n}^\d = \xi_{n-1}^\d$ and $x_{n}^\d = x_{n-1}^\d$. Consequently $t_n^\d = 0$, $m_n^\d =0$ and therefore $\beta_n^\d=0$. Hence, executing the remaining steps does not contribute to the progress toward the desired solution. In practice, this means that when implementing the algorithm, it suffices to choose $i_n \in I_n(y^\d)$
at random according to a uniform distribution. We present Algorithm \ref{shbnoisy} in its current form because this formulation is more convenient for the convergence analysis.
}
\end{remark}

For Algorithm \ref{shbnoisy}, due to the randomness of the indices $i_n$, the iterates $\xi_n^\delta$ and $x_n^\delta$ are stochastic. To facilitate the convergence analysis of Algorithm \ref{shbnoisy}, we first set up the underlying probabilistic framework. Let $\{{\cal F}_n: n \ge 0\}$ denote the natural filtration associate with Algorithm \ref{shbnoisy}, where $\mathcal{F}_n:=\sigma\left(i_0, \cdots ,i_{n - 1} \right)$ is the $\sigma$-algebra generated by the sequence of random indices $i_0, \cdots, i_{n-1}$. Let $\mathbb{E}$ denote the expectation with respect to this filtration; see \cite{Bhat_2007}. By the tower property of conditional expectation, for any random variable $\phi$, we have
\[
\mathbb{E}[\phi] = \mathbb{E}\left[ \mathbb{E}\left[\phi | \mathcal{F}_n \right] \right].
\]
This property will be repeatedly used in our subsequent analysis. In the following result, we show that Algorithm \ref{shbnoisy} is well-defined, in the sense that the iterates $x_n^\delta$ remain in the ball $B_{2\rho}(x_0)$ for all $n\ge 0$, and that the algorithm terminates in finitely many steps almost surely.

\begin{proposition}\label{smonotoniciy}
Let Assumptions \ref{5asumption theta} and \ref{5asumption operator} hold. Assume that
\begin{equation*}
0<\mu_0<4\sigma \quad \mbox{ and } \quad \tau  > (1 + \eta)/(1 - \eta).
\end{equation*}
Then $x_n^\delta \in {B_{2\rho} }\left( {{x_0}} \right)$ for  $n\ge 0$ and, for any solution $\hat x$ of (\ref{nonlinear equation}) in ${B_{2\rho} }\left( {{x_0}} \right)\cap {\rm dom}\left( \Theta  \right)$, there holds
\begin{equation}\label{bregdes1}
\fl D_{\xi_{n + 1}^\delta } \Theta \left(\hat x, x_{n + 1}^\delta\right) - D_{\xi_n^\delta} \Theta \left(\hat x, x_n^\delta\right) \le - c_0 t_n^\delta \left\|\mr_n^\delta\right\|^r - \frac{1}{2} \upsilon_1 \d_{i_n} \beta_n^\d \|\m_n^\d\|^2
\end{equation}
for all integers $n\ge 0$, where $c_0:=\left(1 - \frac{\mu_0}{4\sigma}\right)\left(1 - \eta  - \frac{1 + \eta}{\tau}\right)>0$. Moreover, Algorithm \ref{shbnoisy} terminates after finite steps almost surely.
\end{proposition}

\begin{proof}
We first inductively show that
\[
x_n^\delta \in B_{2\rho}\left(x_0\right) \quad \text{and} \quad D_{\xi_{n + 1}^\delta} \Theta \left(\hat x, x_{n + 1}^\delta\right) \le D_{\xi _0} \Theta \left(\hat x, x_0\right)
\]
for $n\ge 0$, where $\hat x$ denotes any solution of (\ref{nonlinear equation}) in $B_{2\rho}(x_0)\cap \text{dom} (\Theta)$. Since $x_0^\delta=x_0$, it is trivial for $n=0$. Assume that the results hold for $0\le n\le k$. By using (\ref{betan}), we can derive that
\begin{eqnarray*}
& D_{\xi_{k + 1}^\delta}\Theta \left(\hat x, x_{k + 1}^\delta\right) - D_{\xi_k^\delta} \Theta \left(\hat x, x_k^\delta\right)\\
& \le  - \left(1 - \frac{\mu_0}{4\sigma}\right) t_k^\delta \left((1 - \eta) \left\|\mr_k^\delta\right\| - (1 + \eta) \delta_{i_k}\right) \left\|\mr_k^\delta\right\|^{r - 1} + h_k\left(\beta_k^\delta\right).
\end{eqnarray*}
Further, if $\left\|\mr_k^\delta \right\|> \tau \delta _{i_k}$, it follows that
\begin{eqnarray}\label{Theta0}
& D_{\xi_{k + 1}^\delta}\Theta \left(\hat x, x_{k + 1}^\delta\right) - D_{\xi_k^\delta}\Theta \left(\hat x, x_k^\delta\right) \nonumber \\
& \le - \left(1 - \frac{\mu _0}{4\sigma}\right) \left(1 - \eta  - \frac{1 + \eta}{\tau}\right) t_k^\delta \left\| \mr_k^\delta\right\|^r + h_k\left(\beta_k^\delta\right),
\end{eqnarray}
which holds automatically when $\|\mr_k^\delta\|\le \tau \delta_{i_k}$, since (\ref{tndelta}) forces $t_k^\delta=0$.
We next claim that
\begin{equation}\label{hk}
h_k(\beta_k^\d) \le -\frac{1}{2} \upsilon_1 \d_{i_k} \beta_k^\d \|\m_k^\d\|^2.
\end{equation}
When $\|\m_k^\d\| \le \upsilon_0 \d_{i_k}$ or $\tilde \gamma_k^\d - \frac{t_k^\d}{2\sigma} \l \g_k^\d, \m_k^\d\r \ge - \upsilon_1 \d_{i_k} \|\m_k^\d\|^2$, we have $\beta_k^\d = 0$ and thus (\ref{hk}) holds trivially. Otherwise, by the definition of $\beta_k^\d$ we have $0 \le \beta_k^\d \le (t_k^\d \l \g_k^\d, \m_k^\d\r - 2 \sigma \tilde \gamma_k^\d)/\|\m_k^\d\|^2$ and therefore
\begin{eqnarray*}
h_k(\beta_k^\d) & = \frac{1}{4\sigma} \left(\beta_k^\d\right)^2 \|\m_k^\d\|^2
+ \beta_k^\d \left(\tilde \gamma_k^\d - \frac{1}{2\sigma} t_k^\d \left\l \g_k^\d, \m_k^\d \right\r\right) \\
& \le \frac{1}{2} \beta_k^\d \left(\tilde \gamma_k^\d - \frac{1}{2\sigma} t_k^\d \left\l \g_k^\d, \m_k^\d \right\r\right) \\
& \le - \frac{1}{2} \upsilon_1 \d_{i_k} \beta_k^\d\|\m_k^\d\|^2,
\end{eqnarray*}
which shows (\ref{hk}) again. Combining (\ref{hk}) with (\ref{Theta0}) gives (\ref{bregdes1}) for $n = k$ immediately. By using the induction hypothesis, (\ref{bregdes1}) particularly yields
\begin{equation*}
D_{\xi_{k+1}^\delta}\Theta(\hat x, x_{k+1}^\delta) \le D_{\xi_k^\delta}\Theta(\hat x, x_k^\delta) \le D_{\xi _0}\Theta (\hat x, x_0).
\end{equation*}
Taking $\hat x=x^\dag$, we further get that
\[
D_{\xi_{k+1}^\delta} \Theta(x^\dag, x_{k+1}^\delta) \le D_{\xi_k^\delta} \Theta(x^\dag, x_k^\delta) \le D_{\xi_0} \Theta(x^\dag, x_0) \le \sigma \rho ^2,
\]
which, together with the strong convexity of $\Theta$, implies that $\left\|x^\dag - x_{k+1}^\delta\right\| \le \rho $ and $\left\|x^\dag - x_0\right\| \le \rho$. Thus, $\left\|x_{k+1}^\delta - x_0\right\| \le 2\rho $, i.e., $x_{k+1}^\delta \in B_{2\rho}(x_0)$, which completes the inductive proof. Consequently, (\ref{bregdes1}) holds for all $n\ge 0$.

Finally we prove $\mathbb{P}\left(n_\delta = \infty\right) = 0$. To see this, we consider the event
\[
\fl \Psi  := \left\{\max\left\{\left\|F_i\left(x_n^\delta\right) - y_i^\delta\right\| - \tau \delta _i: i \in \left\{ 1, \cdots, N\right\} \right\} > 0~\text{or}~\beta_n^\delta>0~\text{for all}~n \ge 0 \right\}.
\]
It suffices to show $\mathbb{P}\left( \Psi  \right) = 0$. By taking expectation of (\ref{bregdes1}), and rearranging the terms, we have
\[
\fl \mathbb{E}\left[ c_0 t_n^\delta \left\|\mr_n^\delta\right\|^r
+ \frac{1}{2} \upsilon_1 \d_{i_n} \beta_n^\d \|\m_n^\d\|^2 \right]
\le \mathbb{E}\left[D_{\xi _n^\delta }\Theta \left(\hat x, x_n^\delta\right)\right] - \mathbb{E}\left[D_{\xi_{n + 1}^\delta}\Theta \left(\hat x, x_{n + 1}^\delta\right) \right]
\]
and thus, for any integer $k\ge 0$, there follows
\begin{equation}\label{Theta1}
\fl \sum_{n = 0}^k \mathbb{E}\left[c_0 t_n^\delta \left\|\mr_n^\delta\right\|^r
+ \frac{1}{2} \upsilon_1 \d_{i_n} \beta_n^\d \|\m_n^\d\|^2 \right]
\le \mathbb{E}\left[D_{\xi _0}\Theta(\hat x, x_0)\right] = D_{\xi_0}\Theta(\hat x, x_0) < \infty.
\end{equation}
In view of the definition of $t_n^\delta$ and Assumption \ref{5asumption operator}(c), there holds
\[
\fl t_n^\delta \ge \tilde \mu \left\|\mr_n^\delta\right\|^{2 - r} \chi_{\left\{\left\|\mr_n^\delta\right\| > \tau \delta_{i_n} \right\}}~~~\text{with}~~~\tilde \mu  = \min \left\{ {\frac{{{\mu _0}}}{{{B_0^2}}}\left( {1 - \eta  - \frac{{1 + \eta }}{\tau }} \right),{\mu _1}} \right\},
\]
where $\chi _{{\rm A}}$ denotes the characteristic function of an event ${\rm A}$, i.e., ${\chi _{\rm A}}\left( w \right) = 1$ if $w \in {\rm A}$ and $0$ otherwise.
By the tower property of the expectation and the fact that the index $i_n$ is selected from $\left\{ {1, \cdots, N} \right\}$  via a uniform distribution, we further obtain
\begin{eqnarray*}
\mathbb{E}\left[t_n^\delta \left\|\mr_n^\delta\right\|^r\right]
& \ge \tilde \mu \mathbb{E}\left[\left\|\mr_n^\delta\right\|^2 \chi_{\left\{\left\|\mr_n^\delta\right\| > \tau \delta_{i_n} \right\}} \right] \\
& = \tilde \mu \mathbb{E}\left[\mathbb{E}\left[\left\|\mr_n^\delta\right\|^2 \chi_{\left\{\left\|\mr_n^\delta \right\| > \tau \delta_{i_n} \right\}} \big|\mathcal{F}_n \right] \right]\\
& = \frac{\tilde \mu}{N} \mathbb{E}\left[\sum_{i = 1}^N \left\|F_i\left(x_n^\delta\right) - y_i^\delta\right\|^2 \chi_{\left\{\left\|F_i\left(x_n^\delta\right) - y_i^\delta\right\| > \tau \delta_i \right\}} \right].
\end{eqnarray*}
On the other hand, from the definition of $\beta_n^\delta$, it follows that
\begin{eqnarray*}
\fl \quad \beta _n^\delta = \min\left\{\frac{t_n^\d \l \g_n^\d, \m_n^\d\r - 2\sigma \tilde \gamma_n^\d}{\|\m_n^\d\|^2}, \beta\right\} \chi_{\left\{\|\m_n^\delta\| > \upsilon_0 \delta_{i_n} ~\text{and}~ \tilde \gamma_n^\delta -\frac{t_n^\d}{2\sigma} \l \g_n^\d, \m_n^\d\r < - \upsilon_1 \d_{i_n} \|\m_n^\delta\|^2 \right\}} \\
\fl \qquad \ge \min\left\{2\sigma \upsilon_1 \d_{i_n}, \beta\right\} \chi_{\left\{\beta_n^\d >0\right\}}
\ge \tilde \beta \chi_{\left\{\beta_n^\d >0\right\}}
\end{eqnarray*}
with $\tilde \beta :=\min \left\{2\sigma \upsilon_1\d_{\min}, \beta\right\} > 0$, where ${\delta _{\min }} := \min \left\{ {{\delta _i}:i = 1, \cdots, N} \right\}>0$. Thus, we can deduce that
\[
\beta_n^\d \|\m_n^\d\|^2 \ge \tilde \beta \upsilon_0^2 \delta_{i_n}^2 \chi_{\left\{\beta_n^\d >0\right\}}.
\]
Based on the above estimates, it then follows from (\ref{Theta1}) that
\begin{eqnarray*}
\fl D_{\xi _0}\Theta(\hat x, x_0) \\
\fl \ge \sum_{n = 0}^k \mathbb{E}\left[\frac{c_0 \tilde \mu}{N} \sum_{i = 1}^N \left\| {{F_i}\left( {x_n^\delta } \right) - y_i^\delta } \right\|^2 \chi _{\left\{ {\left\| {{F_i}\left( {x_n^\delta } \right) - y_i^\delta } \right\| > \tau {\delta_i}} \right\}}
+ \frac{1}{2} \tilde \beta \upsilon_0^2 \upsilon_1 \d_{i_n}^3 \chi _{\left\{\beta_n^\d >0\right\}} \right] \\
\fl \ge \sum_{n = 0}^k \mathbb{E}\left[\left(\frac{c_0 \tilde \mu}{N} \sum_{i = 1}^N \left\| {{F_i}\left( {x_n^\delta } \right) - y_i^\delta } \right\|^2 \chi _{\left\{ {\left\| {{F_i}\left( {x_n^\delta } \right) - y_i^\delta } \right\| > \tau {\delta _i}} \right\}}
+ \frac{1}{2} \tilde \beta \upsilon_0^2 \upsilon_1 \d_{i_n}^3 \chi _{\left\{\beta_n^\d >0\right\}}\right)\chi_\Psi \right].
\end{eqnarray*}
From the definition of $\Psi$, for any integer $n\ge 0$, either there exists at least one $i \in \left\{ {1, \cdots, N} \right\}$ such that $\left\| {F_i\left(x_{n}^\delta \right)-y_i^\delta} \right\| > \tau \delta_i$ or $\beta_n^\delta>0$ occurs. Consequently
\begin{eqnarray*}
\fl \left(\frac{c_0 \tilde \mu}{N} \sum_{i = 1}^N \left\| {{F_i}\left( {x_n^\delta } \right) - y_i^\delta } \right\|^2 \chi_{\left\{ {\left\| {{F_i}\left( {x_n^\delta } \right) - y_i^\delta } \right\| > \tau {\delta _i}} \right\}} + \frac{1}{2} \tilde \beta \upsilon_0^2 \upsilon_1 \d_{i_n}^3 \chi _{\left\{\beta_n^\d >0\right\}}\right)\chi_\Psi \\
\ge \min\left\{\frac{c_0\tilde \mu}{N} \tau^2 \d_{\min}^2, \frac{1}{2} \tilde \beta \upsilon_0^2 \upsilon_1 \d_{\min}^3\right\} \chi_\Psi = \tilde c \chi_\Psi,
\end{eqnarray*}
where $\tilde c := \min\left\{\frac{c_0\tilde \mu}{N} \tau^2 \d_{\min}^2, \frac{1}{2} \tilde \beta \upsilon_0^2 \upsilon_1 \d_{\min}^3\right\} > 0$. Therefore, we infer that
\begin{eqnarray*}
D_{\xi_0}\Theta (\hat x, x_0) \ge
\sum_{n=0}^k \tilde c \, {\mathbb E}[\chi_\Psi] = \tilde c (k+1) {\mathbb P}(\Psi),
\end{eqnarray*}
from which it follows
\begin{eqnarray*}
{\mathbb P}(\Psi) \le \frac{1}{\tilde c (k+1)} D_{\xi_0}\Theta(\hat x, x_0).
 \end{eqnarray*}
Taking $k \to \infty$ gives $\mathbb{P}\left( \Psi  \right) = 0$, which implies that {\it almost surely} there exists a finite $n$ such that $\left\| {F_i\left(x_{n}^\delta \right)-y_i^\delta} \right\| \le \tau \delta_i$ for $i=1,\cdots,N$ and $\beta_{n}^\delta=0$.
\end{proof}

\subsection{Convergence analysis}

Let $n_\d$ denote the stopping index produced by Algorithm \ref{shbnoisy}. According to Proposition \ref{smonotoniciy}, $n_\d$ is finite almost surely. We emphasis that the stopping index $n_\delta$ is a random variable, which introduces a significant challenge in establishing the convergence of $x_{n_\d}^\d$ to a solution of (\ref{nonlinear equation}) as the noise level $\d$ tends to zero. To demonstrate the almost sure convergence of $x_{n_\delta}^\delta$ as $\delta \to 0$, we examine in this subsection the noise-free counterpart of Algorithm~\ref{shbnoisy} and establish its almost sure convergence. This will be achieved by constructing, in Lemma \ref{descentexact}, an event of probability one and then, in Theorem \ref{sconvergence}, proving that the iterates generated with exact data converge along every sample path within this event.

The counterpart of Algorithm \ref{shbnoisy} with exact data is formulated as follows.

\begin{algorithm}[Adaptive SHB method with exact data]\label{shbexact}
Let $\mu_0>0$, $\mu_1>0$, $0<\beta\le \infty $. Pick $x_0 \in {\rm dom}\left( { \partial \Theta } \right)$ and $\xi _0  \in \partial \Theta\left( {x_0} \right)$. Set $\xi_{-1}=\xi_0$ and $\tilde \gamma _0:=0$.
For $n\ge 0$ do:
\begin{itemize}
\item[\rm{(i)}] Sample  an index  ${i_n} \in \left\{ {1,...,N} \right\}$ randomly via a uniform distribution;

\item[\rm{(ii)}]  Compute $\mr_n:= F_{i_n}(x_n) - y_{i_n}$, $\g_n:=L_{i_n}(x_n)^* J_r^{{\cal Y}_{i_n}}(\mr_n)$, and determine $t_n$ by
\begin{equation*}
t_n = \left\{ \begin{array}{lll}
\min \left\{\frac{\mu_0 (1 - \eta) \|\mr_n\|^r}{\|\g_n\|^2}, \mu_1 \|\mr_n\|^{2 - r} \right\}, & \emph{if}~ \mr_n \ne 0,\\[1mm]
0, & \emph{if}~ \mr_n = 0;
\end{array} \right.
\end{equation*}

\item[\rm{(iii)}] Set $\m_n:= \xi_n - \xi_{n-1}$, compute $\tilde \gamma _n$ by
\begin{eqnarray*}
\tilde \gamma_n = \left\langle \m_n, x_n - x_{n - 1} \right\rangle
- (1 - \eta) t_{n-1} \|\mr_{n - 1}\|^r + \beta _{n - 1} \tilde \gamma_{n - 1}
\end{eqnarray*}
and determine $\beta_n$ by
\begin{eqnarray*}
\fl \qquad \beta_n = \left\{ \begin{array}{lll}
\min \left\{ \frac{t_n \l \g_n, \m_n\r - 2\sigma \tilde \gamma_n}{\|\m_n\|^2}, \beta\right\}, &\emph{if } \m_n \ne 0 \, \& \, \tilde \gamma_n - \frac{t_n}{2\sigma} \l \g_n, \m_n\r <0, \\[1mm]
0, &\emph{otherwise};
\end{array} \right.
\end{eqnarray*}

\item[\rm{(iv)}] Update $\xi _{n + 1}$ and $x _{n + 1}$ by
\begin{eqnarray*}
\xi _{n + 1}  = \xi _n  - t_n \g_n + \beta _n \m_n, \\[1mm]
x_{n + 1}  = \arg \min_{x \in {\cal X}} \left\{\Theta(x) - \left\langle \xi_{n + 1}, x \right\rangle \right\}.
\end{eqnarray*}
\end{itemize}
\end{algorithm}

The following lemma provides an event of probability one on the sample paths, which will serve as the basis for analyzing the almost sure convergence of Algorithm \ref{shbexact}.

\begin{lemma}\label{descentexact}
Let Assumptions \ref{5asumption theta} and \ref{5asumption operator} hold. Consider Algorithm \ref{shbexact} with $0<\mu_0<4\sigma$. Then $x_n \in {B_{2\rho} }\left( {{x_0}} \right)$ for all $n$ and, for any solution $\hat x$ of (\ref{nonlinear equation}) in $B_{2\rho}(x_0)\cap {\rm dom}(\Theta)$, there holds
\begin{equation}\label{monotonicityexact}
D_{\xi_{n + 1}}\Theta(\hat x, x_{n + 1}) \le D_{\xi_n}\Theta(\hat x, x_n)
- c_1 t_n \|\mr_n\|^r
\end{equation}
for all $n\ge 0$, where $c_1: = \left(1 - \frac{\mu_0}{4\sigma} \right)(1 - \eta) > 0$. Moreover, the event
\begin{equation}\label{eventomega}
\mathcal{E}: = \left\{\sum_{n = 0}^\infty\sum_{i = 1}^N \left\|F_i(x_n) - y_i\right\|^2 < \infty\right\}
\end{equation}
occurs almost surely, i.e., ${\mathbb P}\left( \mathcal{E}  \right) = 1$.
\end{lemma}

\begin{proof}
An argument similar to the one used in Lemma \ref{smonotoniciy}  shows that  $x_n \in B_{2\rho}(x_0)$ and (\ref{monotonicityexact}) hold for $n\ge 0$.
The proof is completed by showing that ${\mathbb P}\left(\mathcal{E}\right) = 1$. By taking the expectation of (\ref{monotonicityexact}), we can obtain
\[
\mathbb{E}\left[D_{\xi_n} \Theta \left(\hat x, x_n\right) \right] - \mathbb{E}\left[D_{\xi_{n + 1}} \Theta \left(\hat x, x_{n + 1}\right) \right] \ge c_1 \mathbb{E}\left[t_n \left\|\mr_n\right\|^r \right].
\]
If $\mr_n \ne 0$, then $t_n \ge \mu_2 \|\mr_n\|^{2 - r}$ with $\mu_2 := \min \left\{\mu_0 (1 - \eta)/B_0^2, \mu_1 \right\}$. Consequently, $t_n \left\|\mr_n\right\|^r \ge \mu_2 \left\|\mr_n\right\|^2$. This together with the tower property of the expectation yields
\begin{eqnarray*}
\mathbb{E}\left[D_{\xi_n}\Theta(\hat x, x_n)\right] - \mathbb{E}\left[D_{\xi_{n + 1}}\Theta(\hat x, x_{n + 1}) \right] \\
\ge c_1 \mu_2 \mathbb{E}\left[\|\mr_n\|^2\right] = c_1 \mu_2 \mathbb{E} \left[\mathbb{E}\left[\left\|F_{i_n}(x_n) - y_{i_n}\right\|^2\big| \mathcal{F}_n \right] \right] \\
=\frac{c_1\mu_2}{N}\mathbb{E}\left[\sum_{i = 1}^N \left\|F_i(x_n) - y_i\right\|^2\right].
\end{eqnarray*}
Summing from $n=0$ to $n=\infty$, there follows
\[
\frac{c_1\mu_2}{N}\mathbb{E}\left[\sum_{n = 0}^\infty \sum_{i = 1}^N \left\|F_i(x_n) - y_i\right\|^2\right]
\le D_{\xi_0}\Theta \left(\hat x, x_0\right) < \infty
\]
which implies that
\[
\sum_{n = 0}^\infty \sum_{i = 1}^N \left\|F_i\left(x_n\right) - y_i\right\|^2  < \infty
\]
almost surely. This completes the proof.
\end{proof}

We are now turning to show the almost sure convergence of Algorithm \ref{shbexact}. To this end, we need the following general convergence result.

\begin{proposition}\label{prop_con1}
 Let Assumptions \ref{5asumption theta} and \ref{5asumption operator} hold. Let  $\left\{ {{x_n}} \right\} \subset {B_{2\rho }}\left( {{x_0}} \right)$ and $\left\{ {{\xi_n}} \right\} \subset \mathcal{X}$ be such that
\begin{itemize}
\item[{\rm (i)}] ${\xi _n} \in \partial \Theta \left( {{x_n}} \right)$ for all $n$;

\item[{\rm (ii)}] for any solution $ \hat x$ of (\ref{nonlinear equation}) in  $ {B_{2\rho} }\left( {{x_0}} \right) \cap {\rm dom}\left( \Theta  \right)$ the sequence $\left\{ {{D_{{\xi _n}}}\Theta \left( {\hat x,{x_n}} \right)} \right\}$ is monotonically decreasing;

\item[{\rm (iii)}] ${\lim _{n \to \infty }}\left\| {{F_i}\left( {{x_n}} \right) - {y_i}} \right\| = 0$ for all $i = 1, \cdots ,N$;

\item[{\rm (iv)}] there is a subsequence $\left\{ n_l \right\}$ of integers with $n_l \to \infty$ such that for any solution $\hat x$ of  (\ref{nonlinear equation}) in $ {B_{2\rho} }\left( {{x_0}} \right) \cap {\rm dom}\left( \Theta  \right)$ there holds
\[
\lim_{k \to \infty} \sup_{l\ge k} \left|\left\langle \xi_{n_l} - \xi_{n_k}, x_{n_l} - \hat x \right\rangle\right| = 0.
\]
\end{itemize}
Then there exists a solution $x_*$ of (\ref{nonlinear equation}) in $ {B_{2\rho} }\left( {{x_0}} \right) \cap {\rm dom}\left( \Theta  \right)$ such that
\[
\lim_{n \to \infty} \left\|x_n - x_*\right\| = 0 \quad \text{and} \quad \lim_{n \to \infty} D_{\xi _n}\Theta \left( x_*, x_n\right) = 0.
\]
If, in addition, ${\xi _{n + 1}} - {\xi _n} \in \overline {{\rm Ran}\left( {{L_1}{{\left( {{x^\dag }} \right)}^*}} \right)} \oplus  \cdots  \oplus \overline {{\rm Ran}\left( {{L_N}{{\left( {{x^\dag }} \right)}^*}} \right)}$ for all $n$, then $x_*=x^\dag$.
\end{proposition}

\begin{proof}
Refer to \cite[Proposition 3.6]{JinWang2013}.
\end{proof}

\begin{theorem}\label{sconvergence}
Let Assumptions \ref{5asumption theta} and \ref{5asumption operator} hold. Consider Algorithm \ref{shbexact} with $0<\mu_0<4\sigma$ and $0\le\beta<1$.
Then, there exists a random solution $x_* \in {B_{2\rho} }\left( {{x_0}} \right)\cap  {\rm dom}\left( \Theta  \right)$ of (\ref{nonlinear equation}) such that
\[
\lim_{n \to \infty} \left\|x_* - x_n\right\| = 0 \quad {\rm and} \quad \lim_{n \to \infty} D_{\xi _n}\Theta \left( x_*, x_n\right)=0
\]
almost surely. If $ {{\rm Ran}\left( {{L_i}\left( x \right)^*} \right)} \subset \overline {{\rm Ran}\left( {{L_i}\left( {{x^\dag }} \right)^*} \right)}$ for all $x \in {B_{2\rho }}\left( {{x_0}} \right)$ and $i=1,\cdots,N$, then $x_*=x^\dag$ almost surely.
\end{theorem}

\begin{proof}
Let $\mathcal{E}$ be the event defined in (\ref{eventomega}). It follows from Lemma \ref{descentexact} that ${\mathbb P}\left(\mathcal{E}\right) = 1$. To prove the theorem, it suffices to show that, along any sample path in $\mathcal{E}$, there is a solution $x_*$ of (\ref{nonlinear equation})  in $ {B_{2\rho} }\left( {{x_0}} \right) \cap {\rm dom}\left( \Theta  \right)$ such that ${D_{{\xi _n}}}\Theta \left( {{x_*},{x_n}} \right) \to 0$ as $n \to \infty$. Let $(i_0, i_1, \cdots)$ be an arbitrary but fixed sample path in  $\mathcal{E}$, we next consider the convergence of the sequence $\left\{ {{\xi _n},{x_n}} \right\}$ along this sample path. We will use Proposition \ref{prop_con1}. By the definition of $x_n$ and Lemma \ref{minixi}, the condition (i) holds trivially. From Lemma \ref{descentexact}, the condition (ii) follows. Note that for the term
\[
R_n: = \sum_{i = 1}^N \left\| F_i\left(x_n\right) - y_i\right\|^2,
\]
we have
\[
\sum_{n = 0}^\infty R_n < \infty
\]
and thus $R_n \to 0$ as $n \to \infty$. This implies that $\left\|F_i(x_n) - y_i\right\| \to 0$ as $n\to \infty$ for $i=1,\cdots, N$, which shows the condition (iii) in Proposition \ref{prop_con1}.

What is left is to show the term (iv) in Proposition  \ref{prop_con1}. If $R_n=0$ for some $n$, then  $F_i(x_n) =y_i$ for all $i=1, \cdots, N$, which implies that $x_n$ is a solution of (\ref{nonlinear equation}) in $B_{2\rho}(x_0)\cap  {\rm dom}(\Theta)$. Then, utilizing Lemma \ref{descentexact} with $\hat x=x_n$, we can get that
\begin{equation*}
D_{\xi_m}\Theta(x_n, x_m) \le D_{\xi_n}\Theta(x_n, x_n) = 0
\end{equation*}
for all  $m\ge n$, which, together with the strong convexity of $\Theta$, leads to $x_m=x_n$ for all $m \ge n$. Consequently, $R_m=0$ for all $m \ge n$. Based on these facts, we can choose a strictly increasing subsequence $\{ n_l\}$ of integers by setting $n_0=0$ and, for each $l\ge 1$, letting $n_l$ be the first integer such that
\[
n_l \ge n_{l-1} + 1 \quad \text{and} \quad R_{n_l} \le R_{n_{l - 1}}.
\]
For such sequence, it is easy to show that
\begin{equation}\label{nlxiaoyun}
R_{n_l} \le R_n, \quad \forall 0 \le n \le n_l.
\end{equation}
With the above chosen $\left\{ {{n_l}} \right\}$, we next consider, for any $l>k$,
\[
\left\langle \xi _{n_l} - \xi _{n_k}, x_{n_l} - \hat x \right\rangle  = \sum_{n = n_k}^{n_l - 1} \left\langle \xi_{n + 1} - \xi_n, x_{n_l} - \hat x \right\rangle.
\]
From the definition of $\xi_{n+1}$, it follows that
\begin{equation}\label{xiinduction}
\fl \xi_{n + 1} - \xi_n =  -t_n \g_n + \beta_n \left(\xi_n - \xi_{n - 1}\right) = - \sum_{q = 0}^n \left( \prod_{j = q + 1}^n \beta_j\right) t_q \g_q,
\end{equation}
from which there holds
\begin{eqnarray*}
\fl \left\langle \xi_{n_l} - \xi_{n_k}, x_{n_l} - \hat x \right\rangle
= \sum_{n = n_k}^{n_l - 1} \left\langle - \sum_{q = 0}^n \left(\prod_{j = q + 1}^n \beta_j\right) t_q \g_q, x_{n_l} - \hat x \right\rangle \\
= - \sum_{n = n_k}^{n_l - 1} \sum_{q = 0}^n \left(\prod_{j = q + 1}^n \beta_j \right) t_q\left\langle \g_q, x_{n_l} - \hat x \right\rangle \\
=  - \sum_{n = n_k}^{n_l - 1} \sum_{q = 0}^n \left(\prod_{j = q + 1}^n \beta_j\right) t_q\left\langle J_r^{{\cal Y}_{i_q}} (\mr_q), L_{i_q}(x_q) \left(x_{n_l} - \hat x \right) \right\rangle.
\end{eqnarray*}
Since $t_q \le \mu _1 \|\mr_q\|^{2-r}$, with the help of the Cauchy-Schwarz inequality, we can deduce that
\[
\fl \left|\left\langle \xi_{n_l} - \xi_{n_k}, x_{n_l} - \hat x\right\rangle \right|
\le \mu_1\sum_{n = n_k}^{n_l - 1} \sum_{q = 0}^n \left(\prod_{j = q + 1}^n \beta_j \right) \|\mr_q\| \left\|L_{i_q}(x_q)\left(x_{n_l} - \hat x \right) \right\|.
\]
By utilizing  Assumption \ref{5asumption operator}(c), we get that
\begin{eqnarray*}
\fl \left\| L_{i_q}(x_q)\left(x_{n_l} - \hat x\right) \right\|
&\le \left\| L_{i_q}(x_q)\left(x_{n_l} - x_q\right) \right\| + \left\| L_{i_q}(x_q)\left(x_q - \hat x\right) \right\| \\
&\le (1 + \eta)\left(\left\|F_{i_q}(x_q) - F_{i_q}(x_{n_l}) \right\| + \left\|F_{i_q}(x_q) - y_{i_q}\right\|\right) \\
&\le (1 + \eta)\left(2\left\|F_{i_q}(x_q) - y_{i_q}\right\| + \left\|F_{i_q}(x_{n_l}) - y_{i_q}\right\|\right).
\end{eqnarray*}
Therefore
\begin{eqnarray*}
\left|\left\langle \xi_{n_l} - \xi_{n_k}, x_{n_l} - \hat x \right\rangle\right| \le \Gamma_1 + \Gamma_2,
\end{eqnarray*}
where
\begin{eqnarray*}
\Gamma_1 := 2\mu_1(1+\eta)\sum_{n = n_k}^{n_l-1} \sum_{q=0}^n \left(\prod_{j = q + 1}^n \beta_j\right) \|\mr_q\|^2, \\
\Gamma_2 := \mu_1(1 + \eta)\sum_{n = n_k}^{n_l - 1} \sum_{q = 0}^n \left(\prod_{j = q + 1}^n \beta_j\right) \|\mr_q\| \left\|F_{i_q}\left(x_{n_l}\right) - y_{i_q}\right\|.
\end{eqnarray*}
Since $0 \le \beta_j \le \beta$ for $j\ge 0$, we can estimate $\Gamma _1$ by
\begin{eqnarray*}
\Gamma_1 & \le 2\mu_1 (1 + \eta) \sum_{n = n_k}^{n_l - 1} \sum_{q = 0}^n \beta^{n - q} \left\|F_{i_q}(x_q) - y_{i_q} \right\|^2 \\[1mm]
& \le 2 \mu_1 (1 + \eta) \sum_{n = n_k}^{n_l - 1} \sum_{q = 0}^n \beta^{n - q} \sum_{i = 1}^N \left\| F_i(x_q) - y_i \right\|^2 \\
& = 2 \mu_1 (1 + \eta)\sum_{n = n_k}^{n_l - 1} \sum_{q = 0}^n \beta^{n - q} R_q.
\end{eqnarray*}
By using again $0 \le \beta_j \le \beta$ for $j\ge 0$, and through H\"{o}lder inequality, we can bound $\Gamma_2$ by
\begin{eqnarray*}
\fl \quad  \Gamma_2 \le \mu_1 (1 + \eta) \sum_{n = n_k}^{n_l - 1} \sum_{q = 0}^n \beta^{n - q} \left\|F_{i_q}(x_q) - y_{i_q} \right\| \left\|F_{i_q}(x_{n_l}) - y_{i_q}\right\| \\
\fl \qquad \le \mu_1 (1 + \eta) \sum_{n = n_k}^{n_l - 1} \sum_{q = 0}^n \beta^{n - q} \left(\left\|F_{i_q}(x_q) - y_{i_q} \right\|^2 \right)^{\frac{1}{2}} \left(\left\|F_{i_q}(x_{n_l}) - y_{i_q}\right\|^2\right)^{\frac{1}{2}} \\
\fl \qquad \le \mu_1 (1 + \eta) \sum_{n = n_k}^{n_l - 1} \sum_{q = 0}^n \beta^{n - q} \left(\sum_{i = 1}^N \left\|F_i(x_q) - y_i \right\|^2\right)^{\frac{1}{2}} \left(\sum_{i = 1}^N \left\|F_i(x_{n_l}) - y_i\right\|^2\right)^{\frac{1}{2}} \\
\fl \qquad = \mu_1 (1 + \eta) \sum_{n = n_k}^{n_l - 1} \sum_{q = 0}^n \beta^{n - q} R_q^{\frac{1}{2}} R_{n_l}^{\frac{1}{2}}
\le \mu_1 (1 + \eta) \sum_{n = n_k}^{n_l - 1} \sum_{q = 0}^n \beta^{n - q} R_q,
\end{eqnarray*}
where we used (\ref{nlxiaoyun}) in the last inequality. The combination of the above estimation on ${\Gamma _1}$ and ${\Gamma _2}$ yields
\begin{equation}\label{xinlnk}
\left|\left\langle \xi_{n_l} - \xi_{n_k}, x_{n_l} - \hat x \right\rangle \right| \le 3 \mu_1 (1 + \eta) \sum_{n = n_k}^{n_l - 1} \sum_{q = 0}^n \beta^{n - q} R_q.
\end{equation}
To proceed to prove (iv), let ${S_z} = \sum_{n = 0}^z \sum_{q = 0}^n \beta^{n - q} R_q$ for $z\ge 0$. We can show that the sequence $\{S_z\}$ converges by demonstrating that it is  monotonically increasing with a finite upper bound. Indeed,
\[
S_{z + 1} - S_z = \sum_{q = 0}^{z + 1} \beta^{z+1 - q} R_q  \ge 0,
\]
and due to $0\le \beta<1$, we have the upper bound
\[
S_z = \sum_{q = 0}^z \left(\sum_{n = q}^z \beta^{n - q}\right) R_q \le \frac{1}{1 - \beta} \sum_{q = 0}^z R_q \le \frac{1}{1 - \beta} \sum_{q = 0}^\infty R_q < \infty.
\]
This in particular implies that $\{S_z\}$ is a Cauchy sequence and hence
\[
\lim_{k \to \infty} \sup_{l > k} \sum_{n = n_k}^{n_l - 1} \sum_{q = 0}^n \beta^{n - q} R_q = \lim_{k \to \infty} \sup_{l > k} \left|S_{n_l - 1} - S_{n_k - 1}\right| = 0,
\]
which, together with (\ref{xinlnk}), yields  (iv) in Proposition  \ref{prop_con1}. Therefore, we can use Proposition  \ref{prop_con1} to conclude that there is a solution $x_*$ of (\ref{nonlinear equation}) in $B_{2\rho}(x_0) \cap {\rm dom}(\Theta)$ such that $D_{\xi _n}\Theta(x_*, x_n) \to 0$ as $n \to \infty$. Based on this and the strong convexity of $\Theta$, we conclude that $\|x_* - x_n\| \to 0$ as $n \to \infty$.

By virtue of (\ref{xiinduction}), we have
\[
\xi_{n + 1} - \xi_n \in {\rm Ran}\left(L_{i_0}(x_0)^*\right) \oplus \cdots \oplus {\rm Ran}\left(L_{i_n}(x_n)^* \right).
\]
Under the condition ${\rm Ran}\left(L_i(x)^*\right) \subset \overline{{\rm Ran} \left(L_i(x^\dag)^*\right)}$ for all $x \in B_{2\rho}(x_0)$ and $i=1,\cdots, N$, together with the fact that $x_n \in B_{2\rho }(x_0)$ for $n \ge 0$, we have $\xi_{n + 1} - \xi_n \in \overline{{\rm Ran} \left(L_1(x^\dag)^*\right)} \oplus \cdots \oplus \overline {{\rm Ran}\left(L_N(x^\dag)^*\right)}$ for all $n$. Thus, we can use the second part of Proposition \ref{prop_con1} to show $x_*=x^\dag$ along any sample path in $\mathcal{E}$. Consequently, $x_*=x^\dag$ almost surely.
\end{proof}

\subsection{Regularization property}

Next we establish the almost sure convergence and convergence in expectation of Algorithm \ref{shbnoisy}. Before proceeding further, we show the stability property along every sample path in the following lemma which connects Algorithm \ref{shbnoisy} and Algorithm \ref{shbexact}.


\begin{lemma}\label{sstability}
Let Assumptions \ref{5asumption theta} and \ref{5asumption operator} hold. Assume that $0<\mu_0<4\sigma$ and $\tau > (1 + \eta)/(1 - \eta)$. Consider Algorithm \ref{shbnoisy} and Algorithm \ref{shbexact} along the same sample path. For any fixed sample path, let $\hat n = \lim {\inf _{\delta  \to 0}}{n_\delta }$. Then,
\begin{equation*}
\left\|x_n^\delta  - x_n\right\| \to 0 \quad {\rm and} \quad \left\|\xi_n^\delta  - \xi_n\right\| \to 0 \quad \text{as } \delta\to 0
\end{equation*}
for all $0\le n \le \hat n$.
\end{lemma}

\begin{proof}
Let $\left(i_0, i_1, \cdots\right)$ be an arbitrary fixed sample path. Along this sample path, we will use an induction argument to show for $n \ge 0$ that
\begin{equation}\label{stability0}
\eqalign{
x_n^\delta \to x_n, \quad \xi_n^\delta \to \xi_n, \quad \tilde \gamma_n^\delta \to \tilde \gamma_n, \quad t_n^\delta \g_n^\delta  \to t_n \g_n,\\
t_n^\delta \left\|\mr_n^\delta\right\|^{r - 1} \to t_n \left\|\mr_n\right\|^{r - 1},
\quad \beta _n^\delta \m_n^\d \to \beta_n \m_n, \quad \beta_n^\delta \tilde \gamma_n^\delta \to \beta_n \tilde \gamma_n
}
\end{equation}
as $\delta \to 0$. Since $\xi_{-1}^\delta=\xi_0^\delta=\xi _{ - 1}= \xi_0$, $x_0^\delta=x_0$ and $\tilde \gamma _0^\delta=\tilde \gamma _0=0$, we have $\beta _0^\delta \m_0^\delta = \beta_0 \m_0 = 0$ and $\beta_0^\delta \tilde \gamma_0^\delta  = \beta_0 \tilde \gamma_0 = 0$. Using again $x_0^\delta=x_0$, we have $\mr_0^\delta\to \mr_0$ as $\d \to 0$. Consequently, by using the same argument for deriving (\ref{tsta}) below, we can show that $t_0^\delta \g_0^\delta \to t_0 \g_0$ and $t_0^\delta \|\mr_0^\delta\|^{r - 1} \to t_0 \|\mr_0\|^{r - 1}$ as $\d \to 0$. Thus the results hold for $n=0$. Assuming now the results in (\ref{stability0}) are true for all integers $0\le n\le k$ for some $k \ge 0$, we prove that the results in (\ref{stability0}) also hold for $n = k+1$.

We first show that
\begin{equation}\label{xixsta}
x_{k+1}^\delta  \to {x_{k+1}},  \quad \xi_{k+1}^\delta \to \xi_{k+1}, \quad \tilde \gamma_{k+1}^\delta \to \tilde \gamma_{k+1} \quad \text{as } \delta \to 0.
\end{equation}
 By the induction hypothesis, the definition of $\xi_{k + 1}^\delta$, $x_{k + 1}^\delta$, and (\ref{2.11}) in Lemma \ref{minixi}, we have
 \begin{eqnarray*}
\xi_{k + 1}^\delta  = \xi_k^\delta - t_k^\delta \g_k^\delta  + \beta_k^\delta \m_k^\delta \to \xi_k - t_k \g_k + \beta_k \m_k = \xi_{k + 1},\\
\|x_{k + 1}^\delta - x_{k+1}\| \le \frac{1}{2\sigma} \|\xi_{k + 1}^\delta - \xi_{k + 1}\| \to 0
\end{eqnarray*}
and
\begin{eqnarray*}
\fl \tilde \gamma_{k + 1}^\delta = \left\langle \m_{k + 1}^\delta, x_{k + 1}^\delta - x_k^\delta \right\rangle - (1 - \eta) t_k^\delta \left\|\mr_k^\delta\right\|^r + (1 + \eta) \delta_{i_k} t_k^\delta \left\|\mr_k^\delta\right\|^{r - 1} + \beta_k^\delta \tilde \gamma_k^\delta \\
\to \left\langle \m_{k + 1}, x_{k + 1} - x_k\right\rangle - (1 - \eta) t_k \left\|\mr_k\right\|^r + \beta_k \tilde \gamma_k = \tilde \gamma_{k + 1}
\end{eqnarray*}
as $\delta \to 0$.

Next we prove
\begin{equation}\label{tsta}
t_{k+1}^\delta \g_{k+1}^\delta \to t_{k+1} \g_{k+1}, \quad t_{k+1}^\delta \left\|\mr_{k+1}^\delta\right\|^{r - 1} \to t_{k+1} \left\|\mr_{k+1}\right\|^{r - 1}
\end{equation}
as $\delta \to 0$ by considering two cases on $\mr_{k+1}$. If $\mr_{k+1}=0$, then $t_{k+1}=0$ and by using $0 \le t_{k + 1}^\delta \le \mu_1 \left\|\mr_{k + 1}^\delta\right\|^{2 - r}$, we have
\[
\fl \left\|t_{k + 1}^\delta \g_{k + 1}^\delta - t_{k + 1} \g_{k + 1}\right\|
= \left\|t_{k + 1}^\delta \g_{k + 1}^\delta\right\| \le \mu_1 B_0 \left\|\mr_{k + 1}^\delta\right\|^{r-1} \to \mu_1 B_0 \left\|\mr_{k + 1}\right\|^{r-1} = 0
\]
and
\[
0 \le t_{k + 1}^\delta \left\|\mr_{k + 1}^\delta\right\|^{r - 1} \le \mu_1 \left\|\mr_{k + 1}^\delta\right\|
\to \mu_1 \left\|\mr_{k + 1}\right\| = 0
\]
as $\delta \to 0$. If $\mr_{k+1} \ne 0$, then we must have $\g_{k+1}\ne 0$ because
\begin{eqnarray*}
\left\langle \g_{k + 1}, x_{k + 1} - \hat x \right\rangle \\
= \left\langle J_r^{\mathcal{Y}_{i_{k + 1}}}(\mr_{k + 1}), \mr_{k + 1} + y_{i_{k + 1}} - F_{i_{k + 1}} \left( {x_{k + 1}} \right) - L_{i_{k + 1}}(x_{k + 1}) \left(\hat x - x_{k + 1}\right) \right\rangle \\
\ge \left\|\mr_{k + 1}\right\|^r - \left\|\mr_{k + 1}\right\|^{r - 1} \left\|y_{i_{k + 1}} - F_{i_{k + 1}} (x_{k + 1}) - L_{i_{k + 1}}(x_{k + 1}) (\hat x - x_{k + 1}) \right\| \\
\ge (1 - \eta) \|\mr_{k + 1}\|^r > 0.
\end{eqnarray*}
By the induction hypothesis, the continuity of $F_{i_{k+1}}$, $L_{i_{k+1}}$ and $J_r^{\mathcal{Y}_{i_{k+1}}}$, we have $\mr_{k+1}^\delta \to \mr_{k+1}$ and $\g_{k+1}^\delta \to \g_{k+1}$ as $\delta \to 0$, and hence $\|\mr_{k+1}^\delta\|> \tau \delta_{i_{k+1}}$, and $\g_{k+1}^\delta \ne 0$ for small $\delta$. Consequently, by the definition of $t_{k+1}^\delta$ and $t_{k+1}$, we arrive at $t_{k+1}^\delta \to t_{k+1}$ as $\delta \to 0$. Therefore the results (\ref{tsta}) hold.

Finally we prove
\begin{equation}\label{betasta}
\beta_{k+1}^\delta \m_{k+1}^\delta \to \beta_{k+1} \m_{k+1} \quad \text{and} \quad \beta_{k+1}^\delta \tilde \gamma_{k+1}^\delta  \to \beta_{k+1} \tilde \gamma_{k+1}
\end{equation}
as $\d \to 0$. To show the assertions, we consider three cases. If $\m_{k+1} \ne 0$ and $\tilde \gamma_{k+1} - \frac{t_{k+1}}{2\sigma} \l \g_{k+1}, \m_{k+1}\r < 0$, then, since (\ref{xixsta}) and (\ref{tsta}) imply that
\begin{equation*}
\fl \m_{k+1}^\d \to \m_{k+1} \quad\mbox{and} \quad \tilde \gamma_{k+1}^\d - \frac{t_{k+1}^\d}{2\sigma} \l \g_{k+1}^\d, \m_{k+1}^\d\r \to \tilde \gamma_{k+1} - \frac{t_{k+1}}{2\sigma} \l \g_{k+1}, \m_{k+1}\r < 0
\end{equation*}
as $\d \to 0$, we can conclude that
\begin{equation*}
\|\m_{k+1}^\d\| > \upsilon_0 \d_{i_{k+1}} ~~~\mbox{and} ~~~ \tilde \gamma_{k+1}^\d - \frac{t_{k+1}^\d}{2\sigma} \l \g_{k+1}^\d, \m_{k+1}^\d\r < - \upsilon_1 \d_{i_{k+1}} \|m_{k+1}^\d\|^2
\end{equation*}
for small $\d >0$. Therefore
\begin{eqnarray*}
\beta_{k+1} = \min\left\{\frac{t_{k+1} \l \g_{k+1}, \m_{k+1}\r - 2\sigma \tilde \gamma_{k+1}}{\|\m_{k+1}\|^2}, \beta\right\}, \\
\beta_{k+1}^\d = \min\left\{\frac{t_{k+1}^\d \l \g_{k+1}^\d, \m_{k+1}^\d\r - 2\sigma \tilde \gamma_{k+1}^\d}{\|\m_{k+1}^\d\|^2}, \beta\right\}.
\end{eqnarray*}
Thus, utilizing (\ref{xixsta}) and (\ref{tsta}), we obtain $\beta_{k+1}^\delta\to \beta_{k+1}$ as $\d \to 0$.
Then, by virtue of (\ref{xixsta}), the assertions (\ref{betasta}) follows.

If $\m_{k+1} \ne 0$ and $\tilde \gamma_{k+1} - \frac{t_{k+1}}{2\sigma} \l \g_{k+1}, \m_{k+1}\r \ge 0$, then  by definition we have $\beta_{k+1} =0$. According to the definition of $\beta_{k+1}^\d$ we have either $\beta_{k+1}^\d = 0$ or
\begin{equation*}
0 \le \beta_{k+1}^\d \le \frac{t_{k+1}^\d \l \g_{k+1}^\d, \m_{k+1}^\d\r - 2\sigma \tilde \gamma_{k+1}^\d}{\|\m_{k+1}^\d\|^2}.
\end{equation*}
By using (\ref{xixsta}) and (\ref{tsta}) we have
\begin{eqnarray*}
\frac{t_{k+1}^\d \l \g_{k+1}^\d, \m_{k+1}^\d\r - 2\sigma \tilde \gamma_{k+1}^\d}{\|\m_{k+1}^\d\|^2}
\to \frac{t_{k+1} \l \g_{k+1}, \m_{k+1}\r - 2\sigma \tilde \gamma_{k+1}}{\|\m_{k+1}\|^2} \le 0
\end{eqnarray*}
as $\d \to 0$. Therefore, we can conclude that $\lim_{\d\to 0} \beta_{k+1}^\d = 0 = \beta_{k+1}$. Consequently, by using (\ref{xixsta}) we obtain (\ref{betasta}).

It remains only to consider the case $\m_{k+1}=0$. We again have $\beta_{k+1}=0$ and by (\ref{xixsta})  we also have $\m_{k+1}^\delta \to 0$ as $\delta \to 0$. Due to $0 \le \beta_{k + 1}^\delta  \le \beta$, it follows that
$\lim_{\delta \to 0} \beta_{k + 1}^\delta \m_{k + 1}^\delta = 0 = \beta_{k + 1} \m_{k + 1}$.
Recall that
\[
\tilde \gamma_{k + 1} \ge \gamma_{k + 1} := \left\langle \m_{k + 1}, x_{k + 1} - \hat x \right\rangle  = 0
\]
and by (\ref{xixsta}) we thus have $\lim_{\delta \to 0} \tilde \gamma_{k + 1}^\delta  = \tilde \gamma_{k + 1} \ge 0$. If $\tilde \gamma_{k + 1} =0$, in view of $0 \le \beta_{k + 1}^\delta \le \beta$, we get that
$\lim_{\delta \to 0} \beta_{k + 1}^\delta \tilde \gamma_{k + 1}^\delta = 0 = \beta_{k + 1} \tilde \gamma_{k + 1}$.
If $\tilde \gamma_{k + 1} >0$, then by noting that
\begin{eqnarray*}
\tilde \gamma_{k+1}^\d - \frac{t_{k+1}^\d}{2\sigma} \l \g_{k+1}^\d, \m_{k+1}^\d\r
\to \tilde \gamma_{k+1} - \frac{t_{k+1}}{2\sigma} \l \g_{k+1}, \m_{k+1}\r = \tilde \gamma_{k+1} >0
\end{eqnarray*}
as $\d \to 0$, which implies that
\begin{equation*}
\tilde \gamma_{k+1}^\d - \frac{t_{k+1}^\d}{2\sigma} \l \g_{k+1}^\d, \m_{k+1}^\d\r > -\upsilon_1 \d_{i_{k+1}} \|\m_{k+1}^\d\|^2
\end{equation*}
for small $\d >0$. Consequently, by the definition of $\beta_{k+1}^\delta$, there holds $\beta_{k+1}^\delta=0$ for small $\delta >0$ and thus $\lim_{\delta \to 0} \beta_{k + 1}^\delta \tilde \gamma_{k + 1}^\delta = 0 = \beta_{k + 1} \tilde \gamma_{k + 1}$. We again obtain (\ref{betasta}). The proof is therefore complete.
\end{proof}

Now we turn to show the almost sure convergence and convergence in expectation of Algorithm \ref{shbnoisy}. 

\begin{theorem}\label{6regularization}
Let Assumptions \ref{5asumption theta} and \ref{5asumption operator} hold. Consider Algorithm \ref{shbnoisy} with $0<\mu_0<4\sigma$, $0\le \beta<1$ and $\tau  > (1+\eta)/(1-\eta)$. Let $\{y^{(l)}: = (y_1^{(l)}, \cdots, y_N^{(l)})\}$ be a sequence of noisy data satisfying $\|y_i^{(l)} - y_i\| \le \delta_i^{(l)}$ with $\delta_i^{(l)}>0$ for $i=1, \cdots, N$ and $\delta^{(l)} := \max\{\delta_i^{(l)}: i = 1, \cdots ,N\} \to 0$  as $l \to \infty$. Let ${n_l}:= {n_{{\delta ^{\left( l \right)}}}}$ be the integer produced by Algorithm \ref{shbnoisy}. Then, there exists a random solution $x_* \in {B_{2\rho} }\left( {{x_0}} \right)\cap {\rm dom}\left( \Theta  \right)$ of (\ref{nonlinear equation}) such that
\[\mathop {\lim }\limits_{l  \to \infty}  {{{\left\| {x_{{n_l }}^{{\delta ^{\left( l \right)}}}  - {x_*}} \right\|}}} = 0 ~~{\rm and}~~\mathop {\lim }\limits_{l  \to \infty} {{D_{\xi _{{n_l }}^{{\delta ^{\left( l \right)}}} }}\Theta \left( {{x_*},x_{{n_l }}^{{\delta ^{\left( l \right)}}} } \right)} = 0~~\text{almost surely}\]
and
\[\mathop {\lim }\limits_{l  \to \infty} \mathbb{E}\left[ {{{\left\| {x_{{n_l }}^{{\delta ^{\left( l \right)}}}  - {x_*}} \right\|}^2}} \right] = 0 ~~{\rm and}~~\mathop {\lim }\limits_{l  \to \infty} \mathbb{E}\left[ {{D_{\xi _{{n_l }}^{{\delta ^{\left( l \right)}}} }}\Theta \left( {{x_*},x_{{n_l }}^{{\delta ^{\left( l \right)}}} } \right)} \right] = 0.\]
Further, if $ {{\rm Ran}\left( {{L_i}\left( x \right)^*} \right)} \subset \overline {{\rm Ran}\left( {{L_i}\left( {{x^\dag }} \right)^*} \right)}$ for all $x \in {B_{2\rho }}\left( {{x_0}} \right)$ and $i=1,2,\dots,N$, then $x_*=x^\dag$ almost surely.
\end{theorem}
\begin{proof}
For each noisy data $y^{(l)}: = (y_1^{(l)}, \cdots, y_N^{(l)})$, let $n_l := n(\d^{(l)}, y^{(l)})$ be determined in Algorithm \ref{shbnoisy}. Consider the even $\Psi^{(l)} :=\{n_l = \infty\}$. It follows from Proposition \ref{smonotoniciy} that $\mathbb{P}\left(\Psi^{(l)}\right) = 0$. Let
\[
\tilde \Psi := \bigcup_{l = 1}^\infty \Psi^{(l)}.
\]
By the countable additivity of probability,  $\mathbb{P}(\tilde \Psi) = 0$. On the other hand, from Theorem \ref{sconvergence} it follows that $\lim_{n \to \infty} D_{\xi_n}\Theta \left(x_*, x_n\right)=0$ on an event $\mathcal{E}$, where $\mathbb{P}\left( \mathcal{E}  \right) = 1$ and  $x_*$ is a random solution of (\ref{nonlinear equation}) in $B_{2\rho}(x_0) \cap {\rm dom}(\Theta)$. Let
\[
\Phi=\mathcal{E}\backslash \tilde \Psi.
\]
Then $\mathbb{P}\left(\Phi\right)=1$ and, along any sample path in $\Phi$, we have $n_l<\infty$ for all $l$ and $\mathop {\lim }\nolimits_{n \to \infty } {D_{{\xi _n}}}\Theta \left( {{x_*},{x_n}} \right)=0$.

In the following we first show that along any fixed sample path in $\Phi$, there holds $\mathop {\lim }\nolimits_{l  \to \infty} {{D_{\xi _{{n_l }}^{{\delta ^{\left( l \right)}}} }}\Theta \left( {{x_*},x_{{n_l}}^{{\delta ^{\left( l \right)}}} } \right)} = 0$ by considering two cases.

{\it Case 1}: $\liminf_{l  \to \infty} n_l = \bar n$ for some finite integer $\bar n$. Then, by taking a subsequence if necessary, we have $n_l=\bar n$ for sufficiently large $l$, and, according to the algorithm design, there holds
\[
\left\|F_i\left(x_{\bar n}^{\delta^{(l)}}\right) - y_i^{(l)}\right\| \le \tau \delta ^{(l)}, \quad i=1,\cdots,N.
\]
By taking $l \to \infty$ and using Lemma \ref{sstability}, it follows that ${F_i}\left( {{x_{\bar n}}} \right) = {y_i}$ for $i=1,\cdots,N$. Thus, utilizing Lemma \ref{descentexact} with $\hat x={x_{\bar n}}$ and the strong convexity of $\Theta$, we infer that $x_n=x_{\bar n}$ for $n \ge \bar n$. Since Theorem \ref{sconvergence} gives  $x_n \to x_*$ as $n\to \infty$, we thus get that $x_{\bar n}=x_*$. Then we may use Lemma  \ref{sstability} and lower semi-continuity of $\Theta$ to deduce that
\begin{eqnarray*}
\fl \limsup_{l \to \infty} D_{\xi_{{n_l}}^{\delta^{(l)}}} \Theta \left(x_{*}, x_{n_l}^{\delta^{(l)}}\right)
& = \limsup_{l \to \infty} D_{\xi_{\bar n}^{\delta^{(l)}}} \Theta \left(x_*, x_{\bar n}^{\delta^{(l)}}\right)\\
& = \Theta (x_*)- \liminf_{l \to \infty} \Theta \left(x_{\bar n}^{\delta^{(l)}}\right) - \lim_{l \to \infty} \left\langle \xi_{\bar n}^{\delta^{(l)}}, x_* - x_{\bar n}^{\delta^{(l)}} \right\rangle \\
& \le \Theta(x_*) - \Theta(x_*) = 0,
\end{eqnarray*}
which leads to $D_{\xi_{{n_l}}^{\delta ^{(l)}}} \Theta \left(x_*, x_{n_l}^{\delta^{(l)}} \right)\to 0$ as $l \to \infty$.

{\it Case 2}: $\liminf_{l \to \infty} n_l = \infty$. Then for any fixed integer $n$, we derive from Lemmas \ref{smonotoniciy}, \ref{sstability} and lower semi-continuity of $\Theta$ that
\begin{eqnarray*}
\fl \limsup_{l \to \infty} D_{\xi_{n_l}^{\delta^{(l)}}} \Theta \left(x_*, x_{n_l}^{\delta^{(l)}}\right)
& \le \limsup_{l \to \infty} D_{\xi_n^{\delta^{(l)}}} \Theta \left(x_*, x_n^{\delta^{(l)}}\right)\\
& \le \Theta \left( {{x_*}} \right) - \mathop {\liminf}\limits_{l \to \infty } \Theta \left( {x_n^{{\delta ^{\left( l \right)}}}} \right) - \mathop {\lim }\limits_{l \to \infty } \left\langle {\xi _n^{{\delta ^{\left( l \right)}}},{x_*} - x_n^{{\delta ^{\left( l \right)}}}} \right\rangle\\[2mm]
& \le \Theta \left( {{x_*}} \right) - \Theta \left( {{x_n}} \right) - \left\langle {{\xi _n},{x_*} - {x_n}} \right\rangle
 = {D_{{\xi _n}}}\Theta \left( {{x_*},{x_n}} \right).
\end{eqnarray*}
Letting $n\to \infty$ and using Theorem \ref{sconvergence}, we can conclude that  $ {D_{\xi _{{n_l }}^{{\delta ^{\left( l \right)}}}}}\Theta \left( {{x_{*}},x_{{n_l }}^{{\delta ^{\left( l \right)}}}} \right)\to 0$ as $l \to \infty$.

We have thus shown that $D_{\xi_{n_l}^{\delta^{(l)}}}\Theta(x_*, x_{n_l}^{\delta^{(l)}})\to 0$ as $l\to \infty$ along any sample path in $\Phi$, which implies that   $D_{\xi_{n_l}^{\delta^{(l)}}}\Theta(x_*, x_{n_l}^{\delta^{(l)}})\to 0$  and then, by the strong convexity of $\Theta$, $\|x_{n_l}^{\delta^{(l)}} - x_*\| \to 0$ as $l\to \infty$ almost surely.
Note that
\[{D_{\xi _{{n_l }}^{{\delta ^{\left( l \right)}}} }}\Theta \left( {{x_{*}},x_{{n_l}}^{{\delta ^{\left( l \right)}}} } \right) \le {D_{{\xi _0}}}\Theta \left( {{x_{*}},{x_0}} \right).\]
Thus, we may use the dominated convergence theorem to conclude that  $\mathbb{E}\left[D_{\xi_{n_l}^{\delta^{(l)}}}\Theta(x_*, x_{n_l}^{\delta^{(l)}}) \right] \to 0$ as $l \to \infty$. By the strong convexity of $\Theta$, there follows $\mathbb{E}\left[\|x_{n_l}^{\delta^{(l)}} - x_*\|^2 \right] \to 0$ as $l \to \infty$.
The proof is thus completed.
\end{proof}

\section{Mini-batch variant of Algorithm \ref{shbnoisy}}\label{sect4}
In section \ref{3themethod}, we have developed an adaptive stochastic heavy ball method and established its convergence property under the {\it a posteriori} stopping rule. In this section, we consider a mini-batch version of Algorithm \ref{shbnoisy}. Instead of utilizing a single randomly selected index $i_n$ as in Algorithm \ref{shbnoisy}, we sample a subset of indices $I_n$ from $\{1,\ldots,N\}$, with batch size $b$, randomly without replacement in each step. The mini-batch version of Algorithm \ref{shbnoisy} is presented below.
\begin{algorithm}[Mini-batch version of adaptive SHB method]\label{minibatch_shb}
Let $\mu_0>0$, $\mu_1>0$, $\tau>1$, $\upsilon_0 >0$, $\upsilon_1>0$, $0\le\beta\le \infty$, $b\ge 1$. Pick an initial guess $(x_0, \xi_0) \in {\rm graph}(\partial \Theta)$. Let $\xi_{-1}^\delta=\xi_0^\delta:=\xi_0$ and $x_0^\delta:=x_0$. Set $\mathscr{A}_0(y^\delta):=\left\{ {1, \cdots, N} \right\}$. For $n\ge 0$, do the following:

\begin{itemize}
\item[\emph{(i)}]  Sample an index set  ${I_n} \subset \left\{ {1, \cdots, N} \right\}$ of size $b$ at random via a uniform distribution, and set
 \[\Gamma _n^\delta \left({I_n}\right): = \{i \in I_n:\left\| {F_i\left( {x_n^\delta } \right) - y_i^\delta } \right\| > \tau {\delta _i} \};\]

\item[\emph{(ii)}]  Compute
\begin{eqnarray*}
\g_n^\delta : = \sum\limits_{i \in \Gamma _n^\delta \left({I_n}\right)} {L_i{{\left( {x_n^\delta } \right)}^*}J_r^{{\cal Y}_i}\left(F_i\left(x_n^\delta \right) - y_i^\delta\right)},\\
\Phi _n^\delta : = \sum\limits_{i \in \Gamma _n^\delta \left(I_n\right)} {\left( {\left(1 -\eta \right)\left\|F_i\left( {x_n^\delta } \right) - y_i^\delta \right\| - \left(1 + \eta \right)\delta _i} \right){\left\|F_i\left( {x_n^\delta } \right) - y_i^\delta \right\|^{r - 1}}},\\
\Psi _n^\delta : = \sum\limits_{i \in \Gamma _n^\delta \left( I_n \right)}\left\|F_i\left( {x_n^\delta } \right) - y_i^\delta \right\|^r
\end{eqnarray*}
and determine $t_n^\delta$ by
\[t_n^\delta  = \left\{ \begin{array}{lll}
  \min \left\{ \frac{\mu _0\Phi _n^\delta}{\left\| {\g_n^\delta } \right\|^2},\mu _1\left( {\Psi _n^\delta } \right)^{\frac{2}{r} - 1}\right\},&\emph{if }\Gamma _n^\delta \left(I_n \right) \ne \emptyset, \\
  0,& \emph{if }\Gamma _n^\delta \left( {{I_n}} \right) = \emptyset;
\end{array} \right.\]

\item[\emph{(iii)}] Set $\m_n^\d := \xi_n^\d - \xi_{n-1}^\d$, compute $\tilde \gamma _n^\delta $ by
\[\tilde\gamma _n^\delta : = \left\{ \begin{array}{lll}
  0,&\emph{if } n = 0, \\
\left\langle \m_n^\delta ,x_n^\delta  - x_{n - 1}^\delta \right\rangle  - t_{n - 1}^\delta \Phi _{n - 1}^\delta  + \beta _{n - 1}^\delta \tilde \gamma _{n - 1}^\delta ,&\emph{otherwise};
\end{array} \right.\]
and determine $\beta_n^\d$ by
\[\fl \quad
\beta _n^\delta  = \left\{ \begin{array}{lll}
\min \left\{ \frac{t_n^\d \l \g_n^\d, \m_n^\d\r - 2 \sigma\tilde \gamma_n^\delta}{\|\m_n^\delta\|^2},\beta \right\}, & \emph{if} \left\{ \begin{array}{lll}
  \|\m_n^\delta\| \ge \upsilon_0 \delta _{\min},  \\
  \tilde \gamma_n^\d - \frac{t_n^\d}{2\sigma} \l \g_n^\d, \m_n^\d\r \le - \upsilon_1 \d_{\min} \|\m_n^\d\|^2,
\end{array}  \right. \\
0, & \emph{otherwise};
\end{array} \right.
\]

\item[\emph{(iv)}] Update $\xi _{n + 1}^\delta$ and $x_{n + 1}^\delta$ by
\begin{eqnarray*}
\xi_{n + 1}^\delta  = \xi_n^\delta  - t_n^\delta \g_n^\delta + \beta_n^\delta \m_n^\delta, \\[1mm]
x_{n + 1}^\delta  = \arg \min_{x \in {\cal X}} \left\{\Theta(x) - \left\langle \xi_{n + 1}^\delta, x \right\rangle \right\};
\end{eqnarray*}

\item[\emph{(v)}] Set
\[
{\mathscr{A}_{n + 1}}({y^\delta }): = \left\{ \begin{array}{lll}
{\mathscr{A}_n}({y^\delta })\backslash \left\{ {{I_n}} \right\}, & \emph{if } \Gamma _n^\delta \left(I_n\right) = \emptyset  \emph{ and } \beta_n^\d =0, \\[2mm]
\left\{ {1,\cdots,N} \right\}, & \emph{otherwise};
\end{array} \right.
\]

\item[\emph{(vi)}] Let ${n_\delta}$ be the first integer such that $\mathscr{A}_{n_\delta}(y^\delta)=\emptyset$ and use $x_{n_\delta}^\delta$ as  an approximate solution.
\end{itemize}
\end{algorithm}

For Algorithm \ref{minibatch_shb}, we can use similar arguments for proving Proposition \ref{smonotoniciy} to show that the stopping index $n_\delta$ is finite almost surely. Following the proof of Theorems \ref{sconvergence} and \ref{6regularization}  with only minor adjustments, we can prove the following convergence property of Algorithm \ref{minibatch_shb}.
\begin{theorem}\label{regularization_minibatch}
Let Assumptions \ref{5asumption theta} and \ref{5asumption operator} hold. Consider Algorithm \ref{minibatch_shb} with $0<\mu_0<4\sigma$, $0\le \beta<1$ and $\tau  > (1+\eta)/(1-\eta)$. Let $\{y^{(l)}: = (y_1^{(l)}, \cdots, y_N^{(l)})\}$ be a sequence of noisy data satisfying $\|y_i^{(l)} - y_i\| \le \delta_i^{(l)}$ with $\delta_i^{(l)}>0$ for $i=1, \cdots, N$ and $\delta^{(l)} := \max\{\delta_i^{(l)}: i = 1, \cdots ,N\} \to 0$  as $l \to \infty$. Let ${n_l}:= {n_{{\delta ^{\left( l \right)}}}}$ be the integer produced by Algorithm \ref{minibatch_shb}. Then, there exists a random solution $x_* \in {B_{2\rho} }\left( {{x_0}} \right)\cap {\rm dom}\left( \Theta  \right)$ of (\ref{nonlinear equation}) such that
\[\mathop {\lim }\limits_{l  \to \infty}  {{{\left\| {x_{{n_l }}^{{\delta ^{\left( l \right)}}}  - {x_*}} \right\|}}} = 0 ~~{\rm and}~~\mathop {\lim }\limits_{l  \to \infty} {{D_{\xi _{{n_l }}^{{\delta ^{\left( l \right)}}} }}\Theta \left( {{x_*},x_{{n_l }}^{{\delta ^{\left( l \right)}}} } \right)} = 0~~\text{almost surely}\]
and
\[\mathop {\lim }\limits_{l  \to \infty} \mathbb{E}\left[ {{{\left\| {x_{{n_l }}^{{\delta ^{\left( l \right)}}}  - {x_*}} \right\|}^2}} \right] = 0 ~~{\rm and}~~\mathop {\lim }\limits_{l  \to \infty} \mathbb{E}\left[ {{D_{\xi _{{n_l }}^{{\delta ^{\left( l \right)}}} }}\Theta \left( {{x_*},x_{{n_l }}^{{\delta ^{\left( l \right)}}} } \right)} \right] = 0.\]
Further, if $ {{\rm Ran}\left( {{L_i}\left( x \right)^*} \right)} \subset \overline {{\rm Ran}\left( {{L_i}\left( {{x^\dag }} \right)^*} \right)}$ for all $x \in {B_{2\rho }}\left( {{x_0}} \right)$ and $i=1,2,\dots,N$, then $x_*=x^\dag$ almost surely.
\end{theorem}

\section{Numerical experiments}\label{3numbericalex}
In this section, we provide some numerical simulations to validate the effectiveness of SHB method, i.e., Algorithm \ref{shbnoisy}. To this end, we compare its computational performance with that of the stochastic gradient descent (SGD) method equipped with an {\it a posteriori} stopping rule, which can be stated as follows.

\begin{algorithm}[SGD with noisy data]\label{sgdnoisy}
Let $\mu_0>0$, $\mu_1>0$, and $\tau>1$. Pick an initial guess $(x_0, \xi_0) \in {\rm graph}(\partial \Theta)$. Let $x_0^\delta:=x_0$ and $\xi_0^\delta:=\xi_0$. Set $I_0(y^\delta):=\left\{ {1, \cdots, N} \right\}$. For $n\ge 0$, do the following:

\begin{itemize}
\item[\emph{(i)}]  Sample an index  ${i_n} \in \left\{ {1, \cdots, N} \right\}$ at random via a uniform distribution;

\item[\emph{(ii)}]  Compute $\mr_n^{\delta}:=F_{i_n} \left(x_n^\delta\right) - y_{i_n}^\delta$, $\g_n^\delta:= L_{i_n}\left(x_n^\delta\right)^* J_r^{{\cal Y}_{i_n}} \left(\mr_n^\delta\right)$, and determine $t_n^\delta$ by (\ref{tndelta});

\item[\emph{(iii)}] Update $\xi _{n + 1}^\delta$ and $x_{n + 1}^\delta$ by
\begin{eqnarray*}
\xi_{n + 1}^\delta  = \xi_n^\delta  - t_n^\delta \g_n^\delta, \quad
x_{n + 1}^\delta  = \arg \min_{x \in {\cal X}} \left\{\Theta(x) - \left\langle \xi_{n + 1}^\delta, x \right\rangle \right\};
\end{eqnarray*}

\item[\emph{(iv)}] Set
\[
{I_{n + 1}}({y^\delta }): = \left\{ \begin{array}{lll}
{I_n}({y^\delta })\backslash \left\{ {{i_n}} \right\}, & \emph{if } t_n^\d =0, \\[2mm]
\left\{ {1,\cdots,N} \right\}, & \emph{otherwise};
\end{array} \right.
\]

\item[\emph{(v)}] Let ${n_\delta}$ be the first integer such that $I_{n_\delta}(y^\delta)=\emptyset$ and use $x_{n_\delta}^\delta$ as  an approximate solution.
\end{itemize}
\end{algorithm}

Note that Algorithm \ref{sgdnoisy} is in fact a special case of Algorithm \ref{shbnoisy} with $\beta = 0$. Therefore, the convergence properties established in section \ref{3themethod} can be straightforwardly applied to Algorithm \ref{sgdnoisy}. We compare the SHB method (Algorithm \ref{shbnoisy}) with the SGD method (Algorithm \ref{sgdnoisy}) to demonstrate that incorporating a momentum term can significantly accelerate convergence, thereby justifying the study of the stochastic heavy ball method.

For fair comparison, both the SHB and SGD methods are executed independently 100 times in all following experiments. During the implementation, a key issue is to solve the minimization problem
\begin{equation}\label{minimization problem}
x = \arg \min_{z \in \mathcal{X}} \left\{\Theta(z) - \langle \xi ,z \rangle \right\}
\end{equation}
for any $\xi  \in {\mathcal{X}}$. Below are two common choices of $\Theta$. In case  $\mathcal{X}={L^2}\left( \Omega \right)$ and the sought solution satisfies the constraint $x\in \mathcal{C}$ with $\mathcal{C} \subset \mathcal{X}$ being a closed convex set, we will choose
\begin{equation}\label{thetanonnegative}
\Theta(x) = \frac{1}{{2}} \|x\|_{L^2(\Omega)}^2 + \iota_{\mathcal{C}}(x),
\end{equation}
where $\iota_{\mathcal{C}}$ is the indicator function of $\mathcal{C}$. Obviously, this $\Theta$ satisfies Assumption \ref{5asumption theta} with $\sigma = 1/2$ and the minimizer of (\ref{thetanonnegative}) is given by
$x = P_{\mathcal C}(\xi)$, where $P_{\mathcal C} : {\mathcal X} \to {\mathcal C}$ denotes the orthogonal projection of ${\mathcal X}$ onto ${\mathcal C}$. In particular, when $\mathcal{C} = \left\{x \in \mathcal{X}: x \ge 0~\text{a.e. on}~\Omega\right\}$, the minimizer of  (\ref{minimization problem}) has the explicit formula
\begin{equation*}
x = \max \left\{\xi ,0\right\}.
\end{equation*}
In case $\mathcal{X}={L^2}\left( \Omega \right)$ and the sought solution is piecewise constant, we will take
\begin{equation}\label{thetatv}
\Theta(x) = \frac{1}{2\lambda} \|x\|_{L^2(\Omega)}^2 + |x|_{TV},
\end{equation}
where $|x|_{TV}$ is the total variation of $x$ and $\lambda>0$. The function $\Theta$ in (\ref{thetatv})  satisfies Assumption \ref{5asumption theta} with $\sigma=(2\lambda)^{-1}$, and (\ref{minimization problem}) becomes the total variation denoising problem (\cite{Rudin1992Nonlinear})
\begin{equation}\label{TVSOLVE}
x = \arg \min_{z \in {L^2(\Omega)}} \left\{\frac{1}{2\lambda} \|z - \lambda \xi\|_{L^2(\Omega)}^2 + |z|_{TV} \right\}
\end{equation}
which can be solved by many efficient algorithms (\cite{Jin2025_book}). During the computation, we solve the minimization problem (\ref{TVSOLVE}) by the primal dual hybrid gradient method \cite{Zhu2008An}, which is terminated if the relative duality gap is small than $10^{-3}$ or the number of iterations exceeds 100.

\subsection{The first kind integral equation}

We consider the first kind Fredholm integral equation
\[\int_0^1 {\kappa \left( {s,t} \right)x\left( t \right)dt = y\left( s \right)},~~~~s \in \left[ {0,1} \right]\]
with $\kappa \left( {s,t} \right) = 4{e^{ - {{{\left( {s - t} \right)}^2} \mathord{\left/
 {\vphantom {{{{\left( {s - t} \right)}^2}} {0.01}}} \right.
 \kern-\nulldelimiterspace} {0.01}}}}$. By acquiring the data $y_i=y\left(s_i\right),~i=1,\ldots,N$, at $N$ sample points $s_i$ in $\left[ {0,1} \right]$, the determination of $x\left( t \right)$ reduces to solve the linear system of the form (\ref{nonlinear equation}), that is,
\[{F_i}x: = \int_0^1 {\kappa \left( {{s_i},t} \right)x\left( t \right)dt = {y_i}},\quad i = 1,...,N,\]
where, for each $i$, ${F_i}:{L^2}\left[ {0,1} \right] \to \mathbb{R}$ is a bounded linear operator. Obviously, the condition (\ref{TCC}) in Assumption \ref{5asumption operator} holds with $\eta=0$.

In our numerical simulations, we take $N=300$ and ${s_i} = {{\left(i - 1\right)} \mathord{\left/{\vphantom {\left(i - 1\right)} \left(N - 1\right)} \right.\kern-\nulldelimiterspace}}$ for $i=1,\ldots,N$.  Assume that the sought solution is nonnegative, given by
\[{x^\dag }\left( t \right) = \max \left\{ {40t\left( {t - 0.25} \right)\left( {0.8-t} \right),0} \right\}.\]
Let $y: = \left( {{y_1},...,{y_N}} \right)$ be the exact data with ${y_i} = {F_i}{x^\dag }$. Instead of $y$, we use the noisy data
\[y_i^\delta  = {y_i} + {\delta_{\rm rel}\left\| y \right\|_\infty }{\varepsilon _i},~i=1,\ldots,N,\]
where $\varepsilon _i$ obeys the uniform distribution on $\left[ {-1,1} \right]$ and $\delta_{\rm rel}$ is the relative noise level. Then the noise level is $\delta _i = \delta _{\rm rel}\left\| {y} \right\|_\infty,i=1,2,\ldots,N$. The interval $\left[ {0,1} \right]$ is divided into $N-1$ subintervals of equal size and all integrals are approximated by the trapezoidal rule. To reconstruct the solution $x^\dag$, we take $\Theta$ to be of the form  (\ref{thetanonnegative}) and $\xi_0=x_0=0$. When implementing Algorithm \ref{shbnoisy}, we use the following parameter values
$$\mu_0=0.7,~\mu_1=10^{4},~\tau=1.2.$$
The momentum coefficient $\beta_n^\delta$ is determined by choosing $\beta=0.99$ and considering three distinct pairs of $(v_0,v_1)$. The same parameters are used for executing the SGD method.

\begin{table}[htbp]
\centering
\caption{Numerical results for the first kind integral equation.}
\begin{tabular}{llllll}
\hline
$\delta_{\rm rel}$~&~Method~&~$\left({\upsilon _0},{\upsilon _1} \right)$~&~{\sf iter}~&~{\sf time}(s)&~{\sf error}\\
\hline
\multirow{4}{*}{0.5}&
~SGD~&~-~&~4320~~~~~&~0.0604~&~2.3766e-3~\\
~&~SHB~&~$\left(10^{-5},10^{-4}\right)$~&~2704~~~~~&~0.0403~&~2.8870e-3~\\
~&~SHB~&~$\left(10^{-5},10^{-5}\right)$~&~2937~~~~~&~0.0429~&~2.8553e-3~\\
~&~SHB~&~$\left(10^{-6},10^{-5}\right)$~&~2823~~~~~&~0.0419~&~2.8385e-3~\\
\hline
\multirow{4}{*}{0.1}&
~SGD~&~-~&~14482~~~~~&~0.2545~&~6.3569e-4~\\
~&~SHB~&~$\left(10^{-5},10^{-4}\right)$~&~7056~~~~~&~0.1284~&~6.8774e-4~\\
~&~SHB~&~$\left(10^{-5},10^{-5}\right)$~&~7004~~~~~&~0.1284~&~6.6232e-4~\\
~&~SHB~&~$\left(10^{-6},10^{-5}\right)$~&~6930~~~~~&~0.1250~&~6.6826e-4~\\
\hline
\multirow{4}{*}{0.05}&
~SGD~&~-~&~33387~~~~~&~0.4580~&~2.2069e-4~\\
~&~SHB~&~$\left(10^{-5},10^{-4}\right)$~&~15698~~~~~&~0.2220~&~2.3541e-4~\\
~&~SHB~&~$\left(10^{-5},10^{-5}\right)$~&~16001~~~~~&~0.2276~&~2.7140e-4~\\
~&~SHB~&~$\left(10^{-6},10^{-5}\right)$~&~16082~~~~~&~0.2271~&~2.6492e-4~\\
\hline
\multirow{4}{*}{0.01}&
~SGD~&~-~&~388084~~~~~&~6.8017~&~1.3513e-4~\\
~&~SHB~&~$\left(10^{-5},10^{-4}\right)$~&~45536~~~~~&~0.8444~&~1.7383e-4~\\
~&~SHB~&~$\left(10^{-5},10^{-5}\right)$~&~42222~~~~~&~0.7088~&~1.2919e-4~\\
~&~SHB~&~$\left(10^{-6},10^{-5}\right)$~&~41616~~~~~&~0.6981~&~1.3885e-4~\\
\hline
\multirow{4}{*}{0.005}&
~SGD~&~-~&~1842074~~~~~&~25.8711~&~3.3416e-5~\\
~&~SHB~&~$\left(10^{-5},10^{-4}\right)$~&~95172~~~~~&~1.2881~&~3.9589e-5~\\
~&~SHB~&~$\left(10^{-5},10^{-5}\right)$~&~93240~~~~~&~1.5816~&~3.9741e-5~\\
~&~SHB~&~$\left(10^{-6},10^{-5}\right)$~&~96636~~~~~&~2.0678~&~3.6997e-5~\\
\hline
\end{tabular}
\label{example1}
\end{table}

In Table \ref{example1}, we report the numerical results of the SHB method under various relative noise levels $\delta_{\rm rel}$ and parameter pairs $(v_0,v_1)$. Here, ``{\sf iter}'' records the number of iterations, ``{\sf time}'' represents the computational time and ``{\sf error}'' describes the mean squared relative error $\mathbb{E}\left[\|x_{n_\delta}^\delta  - x^\dag\|^2/\|x^\dag\|^2\right]$. All these quantities are computed as the average over 100 independent runs. For comparison, we also report the results of the SGD method from the same number of independent simulations. These results indicate that both the SHB and SGD methods, coupled with the proposed {\it a posteriori} stopping rule, terminate in finitely many iterations. Moreover, the SHB method, with these values of $(v_0,v_1)$, requires fewer iterations and less computational time than the SGD method, while producing comparable reconstructions with similar relative error. This underscores the striking efficiency of SHB method. To visualize the performance, Figure \ref{error_int} shows the evolution of the squared relative error $\|x_n^\delta - x^\dag\|^2/\|x^\dag\|^2$ with $\delta_{\rm rel}=0.05$ for a single run. The plot clearly demonstrate the acceleration effect of the SHB method.

\begin{figure}[h]
\centering
\includegraphics[scale=0.6]{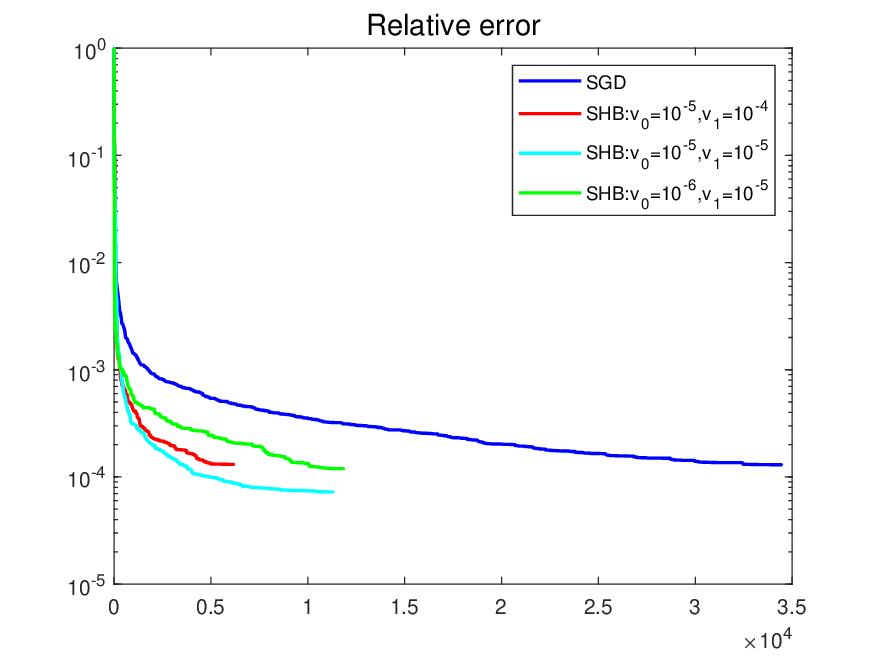}
\vspace*{-5pt}
\caption{Relative error versus the number of iterations with $\delta_{\rm rel}=0.05$.}
\label{error_int}
\end{figure}

To further investigate the convergence behavior of the SHB method, we set $v_0=10^{-6}$, $v_1=10^{-5}$ and display in Figure \ref{boxplot_int} the boxplots of the relative error of the approximate solutions $\|x_{n_\delta}^\delta  - x^\dag\|^2/\|x^\dag\|^2$ and the iteration number $n_\delta$ from 100 simulations under different noise levels. In each box, the central mark denotes the median, while the edges of the box indicate the 25th (bottom) and 75th (top) percentiles. The whiskers extend to the most extreme data points not considered outliers, and outliers (marked by red crosses) are plotted individually. It can be observed that as the noise level decreases, the relative error reduces, and the number of iterations required to satisfy the {\it a posteriori} stopping rule increases. This observation is consistent with the convergence property in Theorem \ref{6regularization}.
\begin{figure}[h]
\centering
\subfigtopskip=2pt
\subfigbottomskip=6pt
\subfigcapskip=-5pt
\subfigure[]{
\includegraphics[scale=0.43]{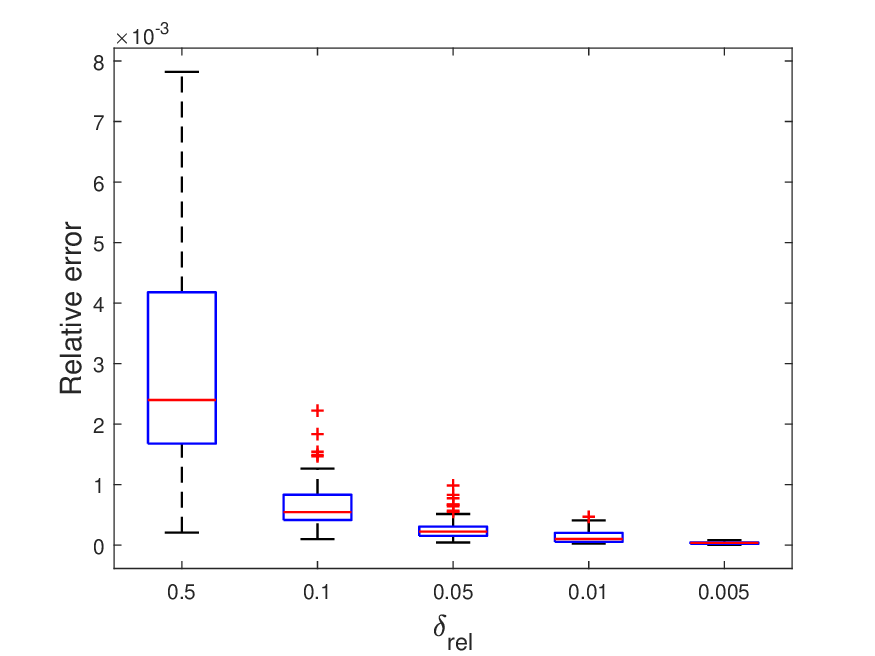}}
\hspace{-1em}
\subfigure[]{
\includegraphics[scale=0.43]{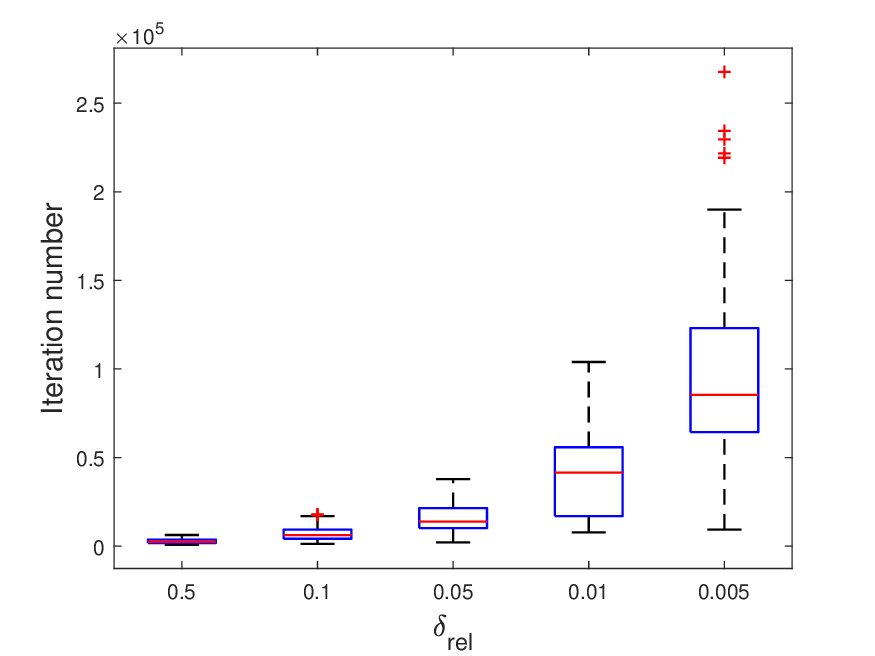}}
\caption{Boxplots of the relative error and iteration number $n_\delta$ from 100 simulations for the first kind integral equation. }
\label{boxplot_int}
\end{figure}

\subsection{Computed tomography}
We next consider the problem of computed tomography (CT), which involves estimating the density of cross sections of an object by measuring the attenuation of X-rays as they pass through the object. This can be mathematically formulated as the reconstruction of a compactly supported function from its Radon transform \cite{Natterer2001}.

In our experiments, the sought image is assumed to be supported on a square domain in $\mathbb{R}^2$ and discretized on a $256\times256$ pixel grid. The image is expressed as a vector $x\in \mathbb{R}^Q$ with $Q=256\times256$  by stacking all its columns. We consider the 2D parallel beam geometry with 45 projection angles equally distributed over 180 degrees, and 360 lines per projection. By using the function \texttt{paralleltomo} in the Matlab package AIR TOOLS \cite{Hansen2012} to generate the discrete problem, we can derive a linear ill-conditioned  system $y=Fx$, where $F$ is a sparse matrix with size $M\times Q$. Here, $M=16200$ and $Q=16384$. Let $F_i$ denote the $i$th row of $F$, and $y_i$ the corresponding component of $y$. This results in a linear system of the form (\ref{nonlinear equation}), that is,
\[
F_i x = y_i, \quad i = 1, \cdots, N
\]
with $N=16200$, where each $F_i: \mathbb{R}^Q \to \mathbb{R}$ is a bounded linear operator. It is evident that the condition (\ref{TCC}) is satisfied with $\eta =0$. In our simulations, the sought image $x^\dag$ is taken to be the Shepp-Logan phantom, depicted in Figure \ref{ct}(a). Let $y_i={F_i}x^{\dag}$, $i=1,\cdots,N$, be the exact data. The noisy data $y_i^{\delta}$ is generated by
\begin{equation}\label{gaussion noise}
y_i^\delta  = {y_i} + {\delta}{\varepsilon _i}, \quad i=1, \cdots, N,
\end{equation}
where each $\varepsilon_i$ follows the standard Gaussian distribution and $\delta$ denotes the noise level. To reconstruct $x^{\dag}$, we use mini-batch variants of the SHB and SGD methods with $b=1800$ and $\Theta$ defined as in (\ref{thetanonnegative}). For Algorithm \ref{minibatch_shb}, we take $\xi_0=x_0=0$ as the initial guess and pick
\[\mu_0=1,~\mu_1=10^{4},~\tau=2.\]
For the coefficient $\beta_n^\delta$, we use $\beta=0.99$ and consider three distinct pairs of parameters $\left({\upsilon _0},{\upsilon _1} \right)$. The same setup is used for the mini-batch SGD method.

\begin{table}[htbp]
\centering
\caption{Numerical results for CT.}
\begin{tabular}{llllll}
\hline
~$\delta$~&~Method~&~$\left({\upsilon _0},{\upsilon _1} \right)$~&~{\sf iter}~&~{\sf time}(s)&~{\sf error}\\
\hline
\multirow{4}{*}{0.5}&
~SGD~&~-~&~9466~~~~~&~29.0723~&~3.7107e-2~\\
~&~SHB~&~$\left(10^{-6},10^{-5}\right)$~&~3215~~~~~&~10.0833~&~3.6919e-2~\\
~&~SHB~&~$\left(10^{-6},10^{-6}\right)$~&~3497~~~~~&~11.0013~&~3.6922e-2~\\
~&~SHB~&~$\left(10^{-7},10^{-6}\right)$~&~3344~~~~~&~10.4534~&~3.6955e-2~\\
\hline
\multirow{4}{*}{0.1}&
~SGD~&~-~&~32202~~~~~&~99.8654~&~2.1520e-2~\\
~&~SHB~&~$\left(10^{-6},10^{-5}\right)$~&~9362~~~~~&~29.6030~&~2.0580e-2~\\
~&~SHB~&~$\left(10^{-6},10^{-6}\right)$~&~9765~~~~~&~31.1790~&~2.0587e-2~\\
~&~SHB~&~$\left(10^{-7},10^{-6}\right)$~&~10191~~~~~&~32.1160~&~2.0578e-2~\\
\hline
\multirow{4}{*}{0.05}&
~SGD~&~-~&~64400~~~~~&~204.8162~&~1.8025e-2~\\
~&~SHB~&~$\left(10^{-6},10^{-5}\right)$~&~16578~~~~~&~53.3191~&~1.7144e-2~\\
~&~SHB~&~$\left(10^{-6},10^{-6}\right)$~&~16466~~~~~&~54.9784~&~1.7163e-2~\\
~&~SHB~&~$\left(10^{-7},10^{-6}\right)$~&~16605~~~~~&~53.4942~&~1.7152e-2~\\
\hline
\end{tabular}
\label{example2}
\end{table}

 Table \ref{example2} summarizes the numerical results of the mini-batch variants of the SHB and SGD methods under various values of $\delta$ and $\left({\upsilon _0},{\upsilon _1} \right)$. The results include the number of iterations ``{\sf iter}'', the computational time  ``{\sf time}'' and the mean squared relative error ``{\sf error}'', denoted by $\mathbb{E}\left[\| x_{n_\delta}^\delta  - {x^\dag }\|^2/\| {x^\dag }\|^2\right]$,  averaged over 100 independent runs. As can be observed, the mini-batch SHB method outperforms the mini-batch SGD method in terms of iteration number and computational time, indicating the acceleration effect of the SHB method. To visualize the performance, Figure \ref{ct} presents the results from a single run with $\delta=0.1$. Figures \ref{ct}(b)-(c) depict the reconstructions obtained by SGD and SHB with $v_0=10^{-7}$ and $v_1=10^{-6}$, respectively. Figure \ref{ct}(d) plots the evolution of the squared relative error $\| x_{n}^\delta - x^\dag\|^2/\|x^\dag\|^2$.

\begin{figure}[h]
\centering
\subfigtopskip=2pt
\subfigbottomskip=6pt
\subfigcapskip=-5pt
\subfigure[]{
\includegraphics[scale=0.43]{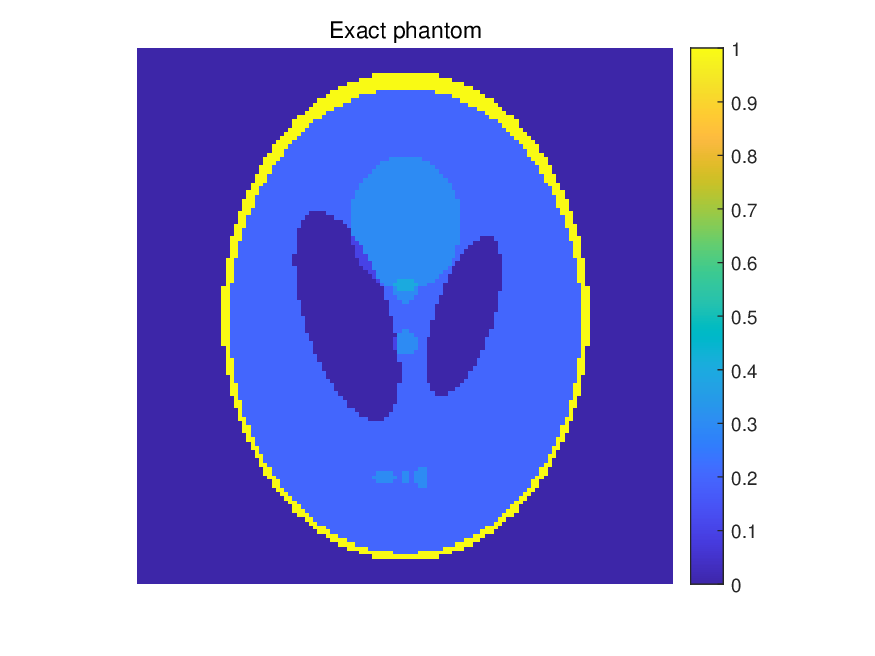}}
 \hspace{-1em}
\subfigure[]{
\includegraphics[scale=0.43]{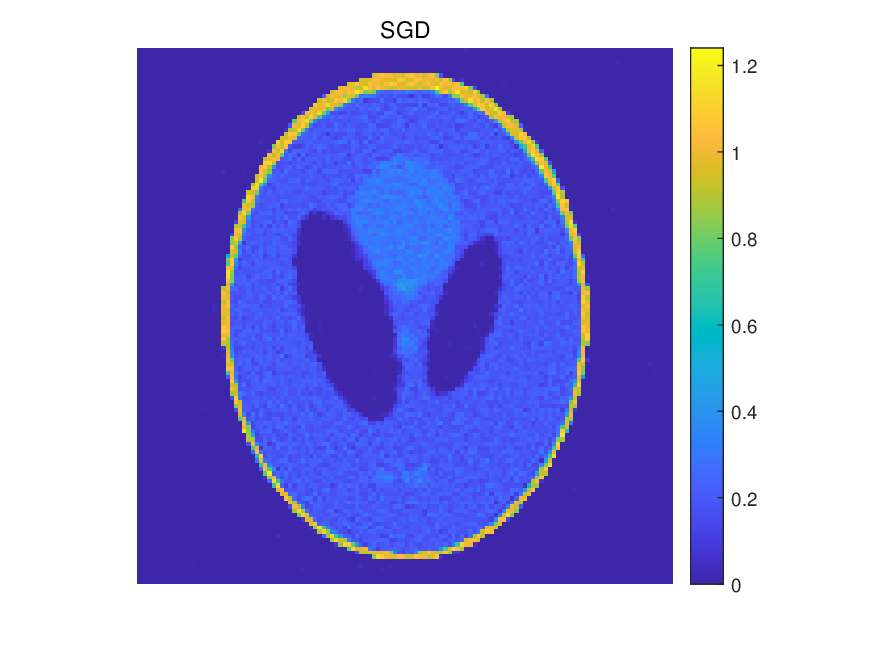}}\\[-2.5mm]
\subfigure[]{
\includegraphics[scale=0.43]{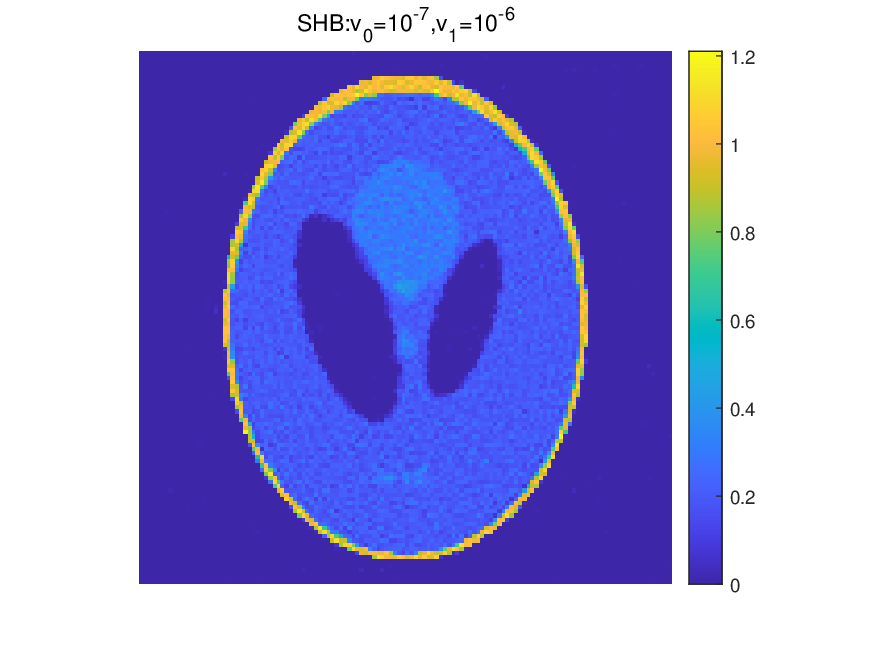}}
 \hspace{-1em}
\subfigure[]{
\includegraphics[scale=0.43]{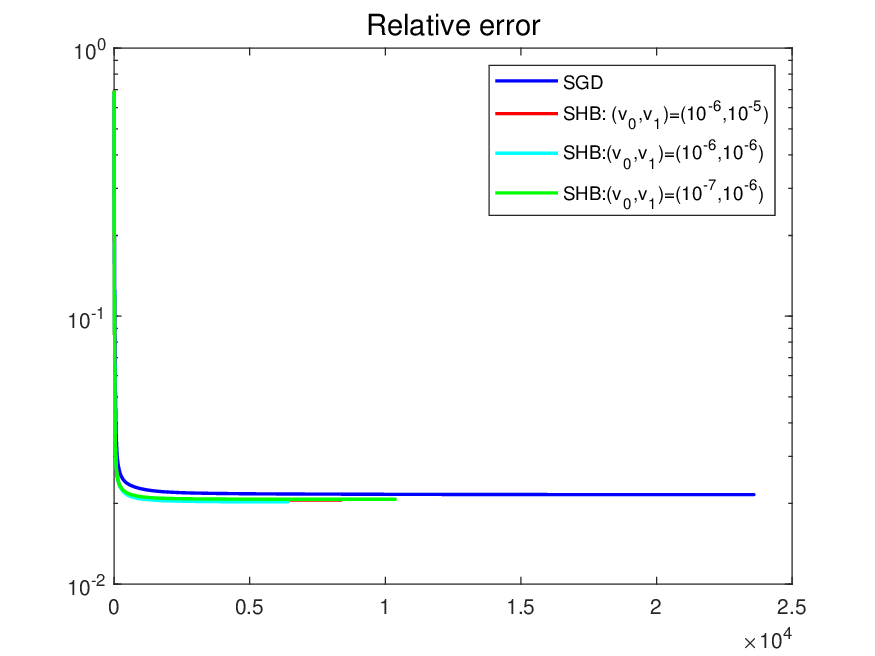}}
\vspace*{-5pt}
\caption{Results for CT with ${\delta}=0.1$. (a) True solution; (b)-(c) reconstructions by SGD and SHB with $v_0=10^{-7}$ and $v_1=10^{-6}$, respectively; (d) evolution of the relative error.}
\label{ct}
\end{figure}

\begin{figure}[h]
\centering
\vspace{-0.35cm}
\subfigtopskip=2pt
\subfigbottomskip=6pt
\subfigcapskip=-5pt
\subfigure[]{
 \includegraphics[scale=0.43]{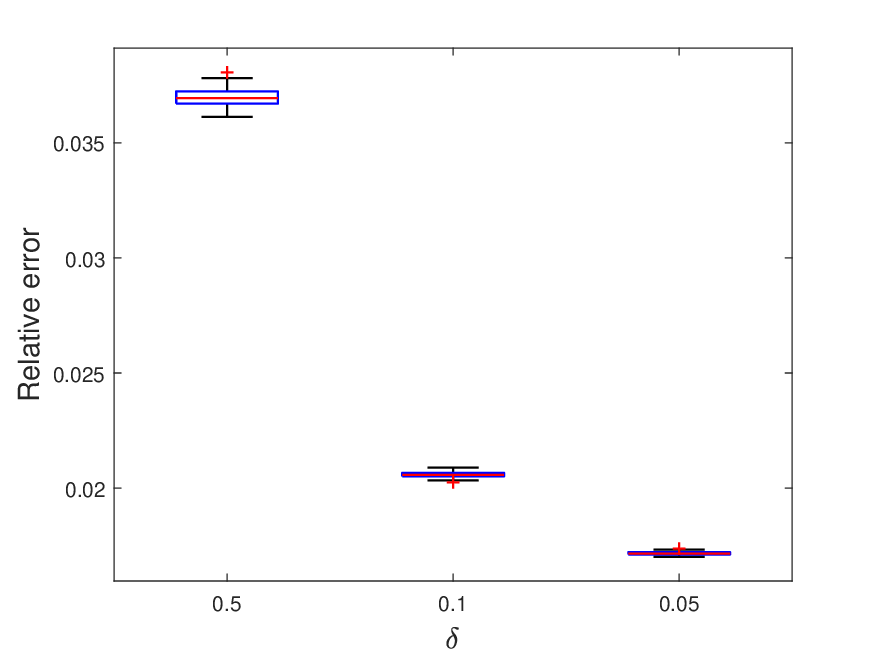}}
 \hspace{-1em}
 \subfigure[]{
 \includegraphics[scale=0.43]{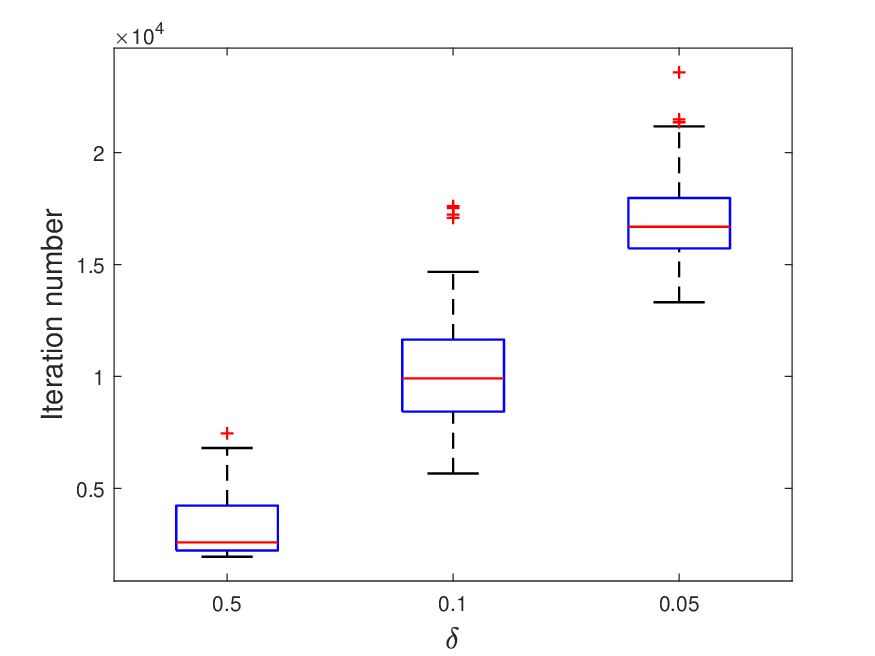}}
\vspace*{-5pt}
\caption{Boxplots of the relative error and iteration number $n_\delta$ from 100 simulations for CT. }
\label{boxplot_ct}
\end{figure}

To investigate the convergence behavior of the mini-batch SHB method, we fix $v_0=10^{-7}$ and $v_1=10^{-6}$. Figure \ref{boxplot_ct} displays the boxplots of the relative error of the approximate solutions  $\| x_{n_\delta}^\delta - x^\dag\|^2/\|x^\dag\|^2$ and the iteration number $n_\delta$, based on 100 simulations under varying noise levels. It can be seen that as the noise level decreases, the relative error reduces; however, the number of iterations required to satisfy the {\it a posteriori} stopping rule increases significantly. This observation further confirms the convergence property in Theorem \ref{regularization_minibatch}.

\begin{table}[htbp]
\centering
\caption{Numerical results of mini-batch SHB method with different $b$.}
\begin{tabular}{lllllll}
\hline
$b$~~~~~~~~~&10~~~~~~~~~~&200~~~~~~~~~~&600~~~~~&1800~~~~&5400~~~&16200\\
\hline
{\sf iter}~&272298~~&34205~~~&17014~&10191~~&6938~~&3894~~\\
{\sf time}(s)~&268.7844~&44.7848~&32.8839~&32.1160~&50.5485~&66.2193\\
{\sf error}~&2.109e-2~&2.074e-2~&2.061e-2~&2.057e-2~&2.079e-2~&2.099e-2\\
\hline
\end{tabular}
\label{test_batch}
\end{table}

 Finally, to examine the impact of the batch size $b$, we apply the mini-batch SHB method with varying values of $b$ to solve the CT problem with $\delta=0.1$. We set $v_0=10^{-7}$ and $v_1=10^{-6}$, while all other settings remain the same as in the previous experiments. For each batch size, 100 independent simulations are conducted. The results for different values of $b$ are presented in Table \ref{test_batch}. We observe that the reconstructed solutions obtained by varying $b$ have similar relative errors. However, the number of iterations decreases significantly as $b$ increases. This trend does not extend to computational time: although the computing time reduces as $b$ grows from 10 to 1800, it rises again for larger $b$. The choice of $b$ for optimal performance warrants further investigation.

\subsection{Schlieren imaging}
Next we consider the problem of reconstructing the 3D pressure fields on cross-section of a water tank generated by an ultrasound transducer from Schlieren data \cite{Hanafy1991}. The objective is to recover a function $f$, supported on a bounded domain $D\subset \mathbb{R}^2$, from
\[
\mathcal{S}f(s,\sigma) = \left(\int_{\mathbb{R}} f(s\sigma + r\sigma^\bot) dr \right)^2, \quad (s,\sigma) \in \mathbb{R} \times \mathbb{S}^1.
\]
In our experiments, we pick $D = [-1, 1]^2$ and collect the data along $N=60$ directions $\sigma_i \in \mathbb{S}^1$, $i=1,\cdots, N$, uniformly distributed on the half circle. 
The reconstruction of $f$ then reduces to solve the nonlinear system of the form (\ref{nonlinear equation}), that is,
\begin{equation}\label{nonsystem}
F_i(f) = y_i, \quad i = 1, \cdots, N,
\end{equation}
where
\[
[F_i(f)](s): = \mathcal{S}f(s, \sigma_i): = \left(R_i f(s)\right)^2 = y_i, \quad i = 1, \cdots, N.
\]
Here, $R_i f(s) := \int_{\mathbb{R}} f(s{\sigma_i} + r\sigma_i^\bot) dr$ denotes the Radon transform in the direction $\sigma_i$. It is known that each $F_i: \mathcal{X} := H_0^1(D) \to \mathcal{Y}_i := L^2([-\sqrt 2, \sqrt 2])$, is Fr\'{e}chet differentiable (\cite{Haltmeier_20071}). The Fr\'{e}chet derivative is given by
\[
F_i'(f) h = 2 R_i f \cdot R_i h, \quad \forall h \in H_0^1(D)
\]
and the adjoint operator $F_i'(f)^*: L^2([-\sqrt{2}, \sqrt{2}]) \to H_0^1(D)$ is given by
\[
F_i'(f)^*g = \left(I - \Delta \right)^{-1} \left( {2 R_i^*\left( {g{R_i}f} \right)} \right),\quad \forall g\in L^2([-\sqrt{2}, \sqrt{2}])
\]
where $I$ is the identity operator, $\Delta$ is the Laplace operator, and $R_i^*:{L^2}([-\sqrt{2}, \sqrt{2}]) \to  {L^2}(D)$ is the adjoint of $R_i$, defined as $(R_i^* g)(x) := g(\langle x, \sigma_i\rangle)$.

In the simulations, we assume that the desired solution $f^\dag$ is piecewise constant; see Figure \ref{schre}(a). Instead of the exact data $y_i = F_i(f^\dag)$, we are only given noisy data
\[
y_i^\delta  = y_i + \delta_{\rm rel} \|y_i\|_{L^2} \varepsilon_i, \quad i = 1, \cdots, N,
\]
where each ${\varepsilon _i}$ follows the normal Gaussian distribution and ${\delta _{\rm rel}}$ denotes the relative noise level. Clearly, for each $i = 1, \cdots, N$ the noise level is $\delta_i = \delta_{\rm rel} \|y_i\|_{L^2}$. For the computation, the domain $D$ is discretized into $110\times 110$  equally sized small squares. To capture the features of $f^\dag$, we take (\ref{thetatv}) with $\lambda=1$. When implementing Algorithm \ref{shbnoisy}, we pick $\xi_0=0.05$ and $f_{0}  = \arg \mathop {\min }\nolimits_{x \in \mathcal{X}} \left\{ {\Theta \left( x \right) - \left\langle {\xi _0 ,x} \right\rangle } \right\}$ as the initial guess and choose
\[\mu_0={0.9}/\lambda, ~\mu_1=10^4, ~\tau=1.5, ~\eta=0.01.\]
To determine $\beta_n^\delta$, we take $\beta=0.99$ and use three different pairs of parameters $(v_0,v_1)$. The same parameter settings are utilized when performing the SGD method.
\begin{table}[htbp]
\centering
\caption{Numerical results for schlieren imaging.}
\begin{tabular}{llllll}
\hline
$\delta_{\rm rel}$~~&~Method~&~$\left({\upsilon _0},{\upsilon _1} \right)$~&~{\sf iter}~&~{\sf time}(s)&~{\sf error}\\
\hline
\multirow{4}{*}{0.05}&
~SGD~&~~-~&~189~~~~~&~0.8810~&~4.2190e-1~\\
~&~SHB~&~$\left(10^{-4},10^{-3}\right)$~&~103~~~~~&~0.6058~&~3.5265e-1~\\
~&~SHB~&~$\left(10^{-4}, 10^{-4}\right)$~&~102~~~~~&~0.6209~&~3.4496e-1~\\
~&~SHB~&~$\left(10^{-5}, 10^{-4}\right)$~&~115~~~~~&~0.7784~&~3.3775e-1~\\
\hline
\multirow{4}{*}{0.02}&
~SGD~&~~-~&~306~~~~~&~3.0774~&~1.0983e-1~\\
~&~SHB~&~$\left(10^{-4},10^{-3}\right)$~~&~164~~~~~&~1.9471~&~9.8755e-2~\\
~&~SHB~&~$\left(10^{-4},10^{-4}\right)$~~&~155~~~~~&~1.7580~&~9.9361e-2~\\
~&~SHB~&~$\left(10^{-5},10^{-4}\right)$~~&~155~~~~~&~1.7436~&~9.7960e-2~\\
\hline
\multirow{4}{*}{0.01}&
~SGD~&~~-~&~1621~~~~~&~12.8245~&~4.4697e-2~\\
~&~SHB~&~$\left(10^{-4},10^{-3}\right)$~&~306~~~~~&~4.1797~&~3.4675e-2~\\
~&~SHB~&~$\left(10^{-4}, 10^{-4}\right)$~&~323~~~~~&~4.3775~&~3.5106e-2~\\
~&~SHB~&~$\left(10^{-5}, 10^{-4}\right)$~&~313~~~~~&~4.4659~&~3.4386e-2~\\
\hline
\multirow{4}{*}{0.005}&
~SGD~&~~-~&~7292~~~~~&~76.3095~&~1.8258e-2~\\
~&~SHB~&~$\left(10^{-4},10^{-3}\right)$~&~592~~~~~&~9.3897~&~1.6102e-2~\\
~&~SHB~&~$\left(10^{-4}, 10^{-4}\right)$~&~591~~~~~&~9.4724~&~1.6112e-2~\\
~&~SHB~&~$\left(10^{-5},10^{-4}\right)$~&~594~~~~~&~9.8862~&~1.6245e-2~\\
\hline
\multirow{4}{*}{0.002}&
~SGD~&~~-~&~75281~~~~~&~838.5848~&~5.5635e-3~\\
~&~SHB~&~$\left(10^{-4},10^{-3}\right)$~&~4493~~~~~&~69.6498~&~4.9028e-3~\\
~&~SHB~&~$\left(10^{-4}, 10^{-4}\right)$~&~4381~~~~~&~72.4039~&~4.8950e-3~\\
~&~SHB~&~$\left(10^{-5},10^{-4}\right)$~&~4099~~~~~&~72.5623~&~4.9172e-3~\\
\hline
\end{tabular}
\label{example3}
\end{table}

Table \ref{example3} summarizes the numerical results of the SHB and SGD methods under different relative noise levels $\delta_{\rm rel}$ and parameters pairs $(v_0,v_1)$. The results record the average number of iterations ``{\sf iter}'', the average computational time ``{\sf time}'' and the average squared relative errors ``{\sf error}'', which are computed over 100 independent simulations. It can be seen that the SHB method, using these three parameter pairs, outperforms the SGD method by substantially reducing the number of iterations and computational time, while producing more accurate reconstructions. This observation clearly demonstrates the superior numerical efficiency of the SHB method. To visualize the performance, we show in Figure \ref{schre}(b)-(d) the results with $\delta_{\rm rel}=0.005$ for a specific run. Figures \ref{schre}(b) and \ref{schre}(c) depict the solutions recovered by SGD and SHB with $v_0=10^{-5}$ and $v_1=10^{-4}$, respectively. The reconstructions by SHB with the other two parameter pairs are similar to that in Figure \ref{schre}(c) and are therefore omitted. Figure \ref{schre}(d) plots the evolution of the relative error $\| x_n^\delta-x^\dag\|^2 /\|x^\dag\|^2$, which indicates the fast convergence of the SHB method.

\begin{figure}[htbp]
\centering
\vspace{-0.35cm}
\subfigtopskip=2pt
\subfigbottomskip=6pt
\subfigcapskip=-5pt
\subfigure[]{
\includegraphics[scale=0.43]{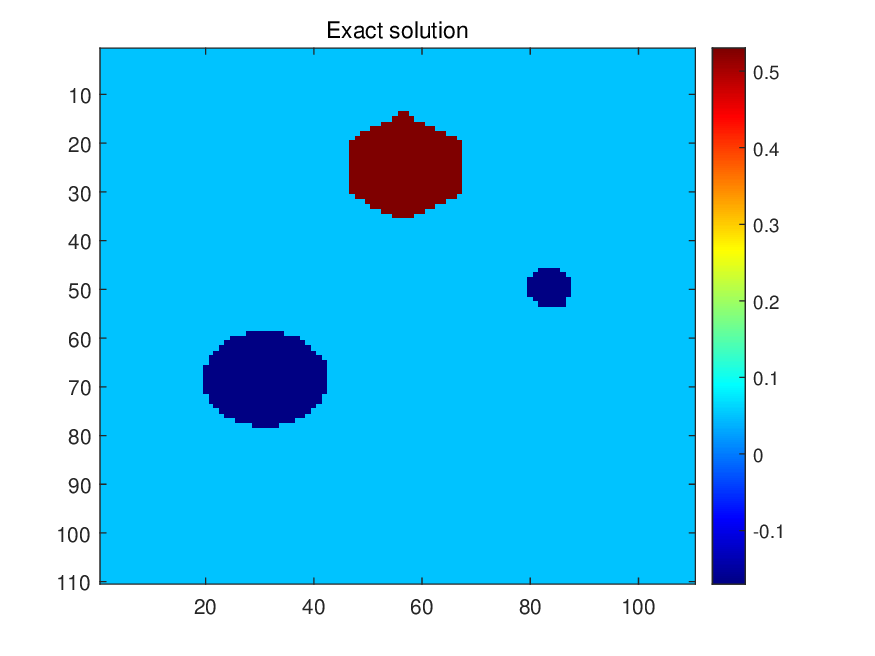}}
\hspace{-1em}
\subfigure[]{
\includegraphics[scale=0.43]{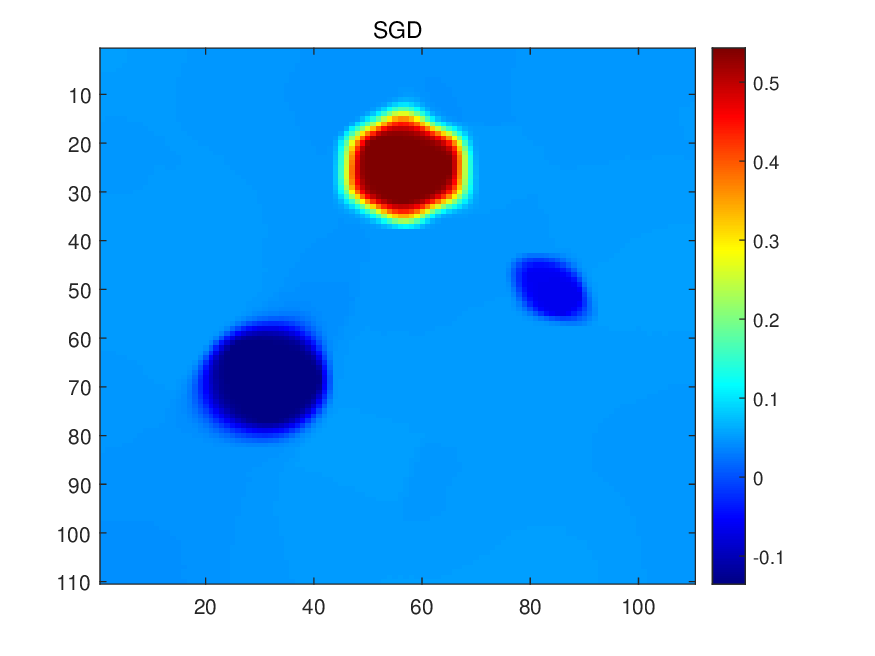}}\\[-2.5mm]
\subfigure[]{
\includegraphics[scale=0.43]{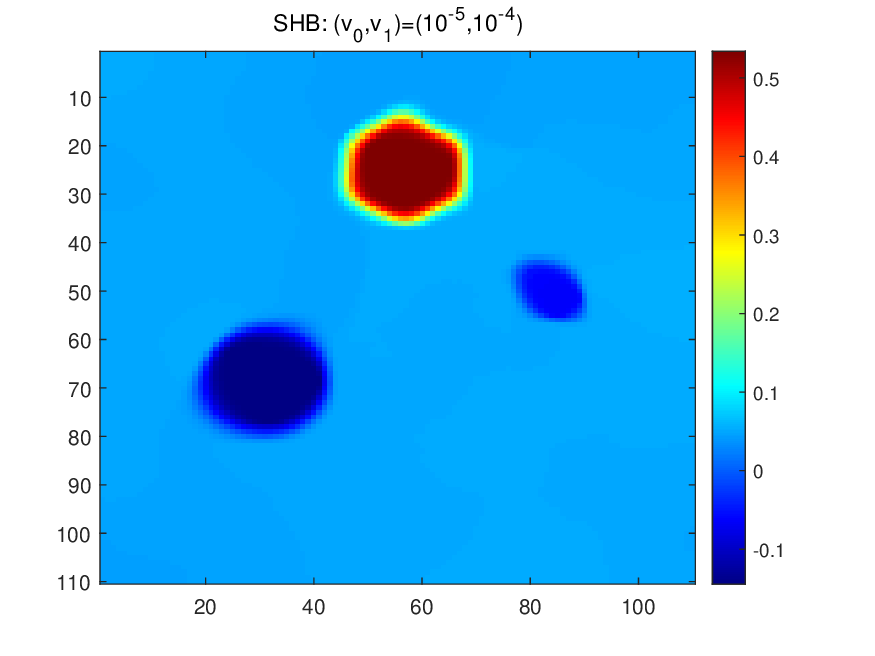}}
\hspace{-1.5em}
\subfigure[]{
 \includegraphics[scale=0.43]{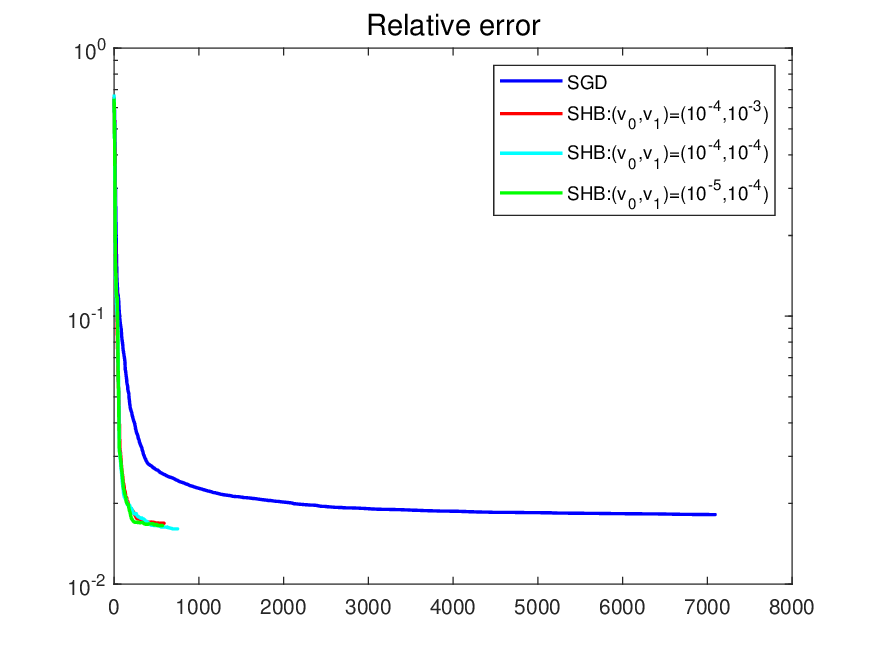}}
\vspace*{-5pt}
\caption{Results for schlieren imaging with $\delta_{\rm rel}=0.005$. (a) Exact solution; (b)-(c) reconstructions of SGD and SHB with $v_0=10^{-5}$ and $v_1=10^{-4}$, respectively; (d) relative error versus the number of iterations.}
\label{schre}
\end{figure}

To further examine the convergence of the SHB method, we fix $v_0=10^{-5}$, $v_1=10^{-4}$ and present, in Figure \ref{boxplots_sch}, the boxplots of the relative error of the approximate solution $\|x_{n_\delta}^\delta-x^\dag \|^2/\|x^\dag\|^2$ and the required number of iterations $n_\delta$, based on 100 independent simulations at different relative noise levels. As observed, when the relative noise level decreases, the reconstruction accuracy improves, reflected in a lower relative error; however, the greater number of iterations is required.  This observation confirms the convergence in Theorem \ref{6regularization}.
\begin{figure}[htbp]
\centering
\vspace{-0.35cm}
\subfigtopskip=2pt
\subfigbottomskip=6pt
\subfigcapskip=-5pt
\subfigure[]{
\includegraphics[scale=0.43]{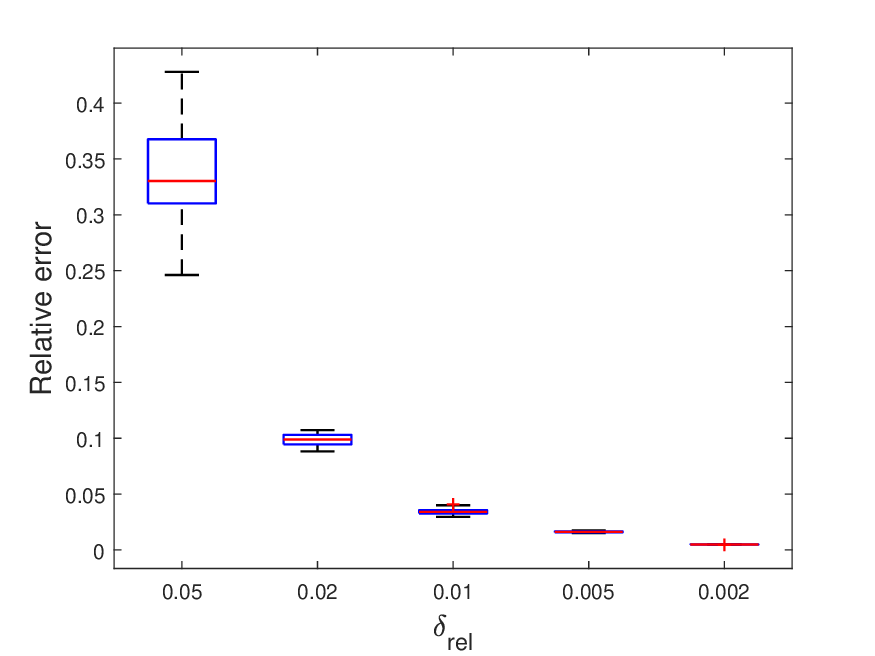}}
\hspace{-1em}
\subfigure[]{
\includegraphics[scale=0.43]{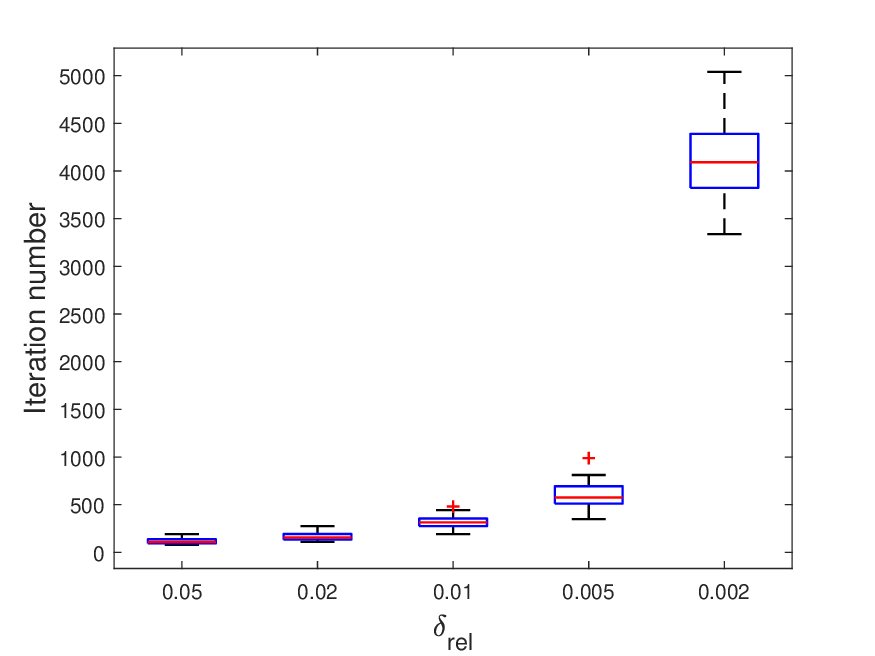}}
\vspace*{-5pt}
\caption{Boxplots of the relative error and iteration number from 100 runs for schlieren imaging.}
\label{boxplots_sch}
\end{figure}

\section{Conclusion}\label{conclusion}
In this work, we have designed an  {\it a posteriori} stopping rule for the stochastic heavy ball method for solving large-scale inverse problems, which avoids the need to calculate the residuals of each equation in the systems and therefore is computationally efficient. To promote fast convergence, an adaptive strategy has been developed to determine the step size and the momentum coefficient. Under some conditions, we have established the almost sure convergence and convergence in expectation of the method under the proposed {\it a posteriori} stopping rule. Further, numerical results demonstrate the effectiveness of the {\it a posteriori} stopping rule and the superiority of the method in terms of the required number of iterations and computational time compared with the SGD method.

\ack{R Gu is partially supported by the National Natural Science Foundation of China (No. 12301535) and China Postdoctoral Science Foundation (No. 2024M750292). Q Jin is partially supported by the Future Fellowship of the
Australian Research Council (FT170100231).}

\section*{References}

\bibliographystyle{unsrt}

\begin{thebibliography}{10}


\bibitem{Bhat_2007}
R. N. Bhattacharya and E. C. Waymire.
\newblock {\em A basic course in probability theory}.
\newblock New York: Springer, 2007.

\bibitem{Bo2011Iterative}
R. I. Bo\c{t} and T. Hein.
\newblock {\em Iterative regularization with a general penalty term-theory and
  application to $L^1$ and TV regularization}.
\newblock Inverse Problems, 28 (2011), 104010.

\bibitem{Bottou_2010}
 L. Bottou.
\newblock {\em Large-scale machine learning with stochastic gradient descent}.
\newblock Proceedings of COMPSTAT2010 (Berlin: Springer), 2010, pp. 177-186.

\bibitem{Bottou_2018}
L. Bottou, F. E. Curtis and J. Nocedal.
\newblock {\em Optimization methods for large-scale machine learning.}
\newblock SIAM Review, 60 (2018), pp. 223-311.

\bibitem{Cioranescu1990Geometry}
 I. Cioranescu.
\newblock {\em Geometry of Banach spaces, duality mappings and nonlinear
  problems}.
\newblock Kluwer Academic Pub, 1990.

\bibitem{Engl1996Regularization}
H. W. Engl, M. Hanke and A. Neubauer.
\newblock {\em Regularization of inverse problems}.
\newblock Kluwer Academic Pub, 1996.


\bibitem{Gu-Fu-2025}
R. Gu, Z. Fu, B. Han, H. Fu.
\newblock  {\em Stochastic gradient descent method with convex penalty for ill-posed problems in Banach space.}
\newblock Inverse Problems, 41 (2025), 055003.

\bibitem{Haltmeier_20071}
 M. Haltmeier, R. Kowar, A. Leit$\tilde{\text{a}}$o and O. Scherzer.
\newblock {\em Kaczmarz methods for regularizing nonlinear ill-posed equations. II. applications}. \newblock Inverse Problems and Imaging, 3 (2007), pp. 507-523.

\bibitem{Hanafy1991}
 A. Hanafy and C. I. Zanelli.
\newblock {\em Quantitative real-time pulsed Schlieren imaging of ultrasonic waves}.
\newblock Proceedings of IEEE Ultrasonics Symposium, 2 (1991), pp. 1223-1227.

\bibitem{Hanke1995A}
M. Hanke, A. Neubauer, and O. Scherzer.
\newblock {\em A convergence analysis of the Landweber iteration for nonlinear ill-posed problems}.
\newblock Numerische Mathematik, 72 (1995), pp. 21-37.

\bibitem{Hansen2012}
 P. C. Hansen and M. Saxild-Hansen.
\newblock {\em AIR tools-a MATLAB package of algebraic iterative reconstruction methods}.
\newblock Journal of Computational and Applied Mathematics, 236 (2012), pp.  2167-2178.



\bibitem{Huang-Jin-Lu-Zhang-2025}
J. Huang, Q. Jin, X. Lu, L. Zhang.
\newblock  {\em On early stopping of stochastic mirror descent method for
ill-posed inverse problems.}
\newblock Numerische Mathematik, 157 (2025), pp. 539-571.





\bibitem{JinB_2019}
B. Jin and X. Lu.
\newblock {\em On the regularization property of stochastic gradient descent}.
\newblock Inverse Problems, 35 (2019), 015004.

\bibitem{JinB_2020}
B. Jin, Z. Zhou and J. Zou.
\newblock {\em On the convergence of stochastic gradient descent for nonlinear ill-posed problems}.
\newblock SIAM Journal on Optimization, 30 (2020), pp. 1421-1450.

\bibitem{Jin-Zhou-zou-2022}
B. Jin, Z. Zhou and J. Zou.
 {\em An analysis of stochastic variance reduced gradient for linear inverse
problems.}
\newblock Inverse Problems, 38 (2022), 025009.



\bibitem{Jin-2025}
Q. Jin.
\newblock {\em Adaptive Nesterov momentum method for solving ill-posed inverse problems.}
\newblock Inverse Problems, 41 (2025), 025005.

\bibitem{Jin2025_book} Q. Jin, {\it Lectures on Nonsmooth Optimization}, Texts in Applied Mathematics, 82.
Springer, Cham, 2025.

\bibitem{Jin-chen-2025}
Q. Jin, L. Chen.
\newblock  {\em Stochastic variance reduced gradient method for linear ill-posed inverse problems.}
\newblock Inverse Problems, 41 (2025), 055014.

\bibitem{Jin-Huang-2024}
Q. Jin and Q. Huang.
\newblock {\em  An adaptive heavy ball method for ill-posed inverse problems.}
\newblock SIAM Journal on Imaging Sciences, 17 (2024), pp. 2212-2241.


\bibitem{Jin-Liu-2025}
Q. Jin and Y. Liu.
\newblock {\em Convergence analysis of a stochastic heavy-ball method for linear ill-posed problems.}
\newblock Journal of Computational and Applied Mathematics, 470 (2025), 116702.

\bibitem{JinQ_2023}
 Q. Jin, X. Lu and L. Zhang.
\newblock {\em Stochastic mirror descent method for linear ill-posed problems in Banach spaces}. \newblock Inverse Problems, 39 (2023), 065010.

\bibitem{JinWang2013}
 Q. Jin and W. Wang.
\newblock {\em Landweber iteration of Kaczmarz type with general non-smooth convex penalty functionals}.
\newblock Inverse Problems, 29 (2013),  pp. 1400--1416.

\bibitem{KNS2008}
 { B. Kaltenbacher, A. Neubauer, and O. Scherzer.}
\newblock {\em Iterative regularization methods for nonlinear ill-posed problems}.
\newblock Walter de Gruyter GmbH \& Co. KG, Berlin, 2008.

\bibitem{Lu_2022}
 S. Lu and P. Math$\acute{\text{e}}$.
\newblock {\em Stochastic gradient descent for linear inverse problems in Hilbert spaces}.
\newblock Mathematics of Computation, 91 (2022), pp. 1763-1788.

\bibitem{Natterer2001}
 F. Natterer.
\newblock {\em The Mathematics of Computerized Tomography}.
\newblock Philadelphia, PA: SIAM, 2001.


\bibitem{Robbins_1951}
 H. Robbins and S. Monro.
\newblock {\em A stochastic approximation method}.
\newblock The Annals of Mathematical Statistics, 22 (1951), pp. 400-407.

\bibitem{Rudin1992Nonlinear}
 L. I. Rudin, S. Osher, and E. Fatemi.
\newblock {\em Nonlinear total variation based noise removal algorithms}.
\newblock Physica D., 60 (1992),  259--268



\bibitem{Schuster2012Regularization}
T. Schuster, B. Kaltenbacher, B. Hofmann, and K. S. Kazimierski.
\newblock {\em Regularization methods in Banach Spaces}.
\newblock De Gruyter, 2012.

\bibitem{Zalinescu2002Convex}
 C. Zalinescu.
\newblock {\em Convex analysis in general vector spaces}.
\newblock World Scientific, 2002.


\bibitem{Zhu2008An}
 M. Zhu and T. Chan.
\newblock {\em An efficient primal-dual hybrid gradient algorithm for total
  variation image restoration}.
\newblock CAM Report, 08-34, UCLA, 2008.
%
%

\end{thebibliography}

\end{document}